\documentclass[11pt]{article}

\usepackage{amsmath,amsthm}

\usepackage{amssymb,latexsym}

\usepackage[mathscr]{euscript}

\usepackage{enumerate,color}

\usepackage{enumerate}

\newcommand{\BB}{{\cal B}}

\newcommand{\EE}{{\cal E}}
\newcommand{\FF}{{\cal F}}

\newcommand{\HH}{{\cal H}}
\newcommand{\TT}{{\cal T}}

\newcommand{\MM}{{\cal M}}

\newcommand{\BR}{{\mathbb R}}
\newcommand{\BX}{{\mathbb X}}

\newcommand{\esssup}{\mathop{\mathrm{ess\,sup}}}

\newtheorem{theorem}{\bf Theorem}[section]
\newtheorem{proposition}[theorem]{\bf Proposition}
\newtheorem{lemma}[theorem]{\bf Lemma}%[subsection]
\newtheorem{corollary}[theorem]{\bf Corollary}

\theoremstyle{definition}
\newtheorem{definition}[theorem]{Definition}
\newtheorem{example}[theorem]{\bf Example}

\newtheorem{remark}[theorem]{Remark}

\numberwithin{equation}{section}

\begin{document}

\title {Schr\"odinger equations with  smooth measure potential and general measure data}
\author {Tomasz Klimsiak\\
{\small Faculty of Mathematics and Computer Science,
Nicolaus Copernicus University} \\
{\small  Chopina 12/18, 87--100 Toru\'n, Poland}\\
{\small E-mail address: tomas@mat.umk.pl}}
\date{}
\maketitle
\begin{abstract}
We study  equations driven by Schr\"odinger operators consisting of  a self-adjoint Dirichlet operator
and  a  singular potential, which belongs to a class of  positive Borel measures
absolutely continuous with respect to a capacity generated by the operator. In particular,
we cover positive potentials exploding on a set of capacity zero. 
The right-hand side of equations is allowed to be   a general bounded Borel measure. 
The class  of self-adjoint Dirichlet operators is quite large. Examples include  integro-differential operators
with the local part of divergence form.  We give a necessary and sufficient condition for the
existence of a solution, and prove some regularity and stability results.
\end{abstract}



%\footnotetext{This work was supported by Polish National Science Centre
%(grant no. 2016/23/B/ST1/01543).}

\section{Introduction}

Let $E$ be a locally compact separable  metric space and $m$ be a
Radon measure on $E$ with full support. Let $A$ be a non-positive
self-adjoint operator on $L^2(E;m)$ generating a Markov semigroup
of contractions  $(T_t)_{t\ge 0}$ on $L^2(E;m)$ (so-called Dirichlet operators). We also assume that 
  that there exists the Green function $G$ for  $-A$
(see Section \ref{sec2.2}). In the paper, we give a necessary and
sufficient condition for the existence of a solution to the
following Schr\"odinger equation
\begin{equation}
\label{eq1.1}
-Au+u\cdot\nu=\mu.
\end{equation}
Here $\nu$ belongs to  the set $S_A$ consisting  of  positive {\em smooth measures}: Borel measures  absolutely continuous with respect to a capacity $Cap_A$
generated by $A$ for which there exists a strictly positive quasi-continuous function $\eta$ such that $\int_E\eta\,d\nu<\infty$.
%(with respect to $A$; see Section \ref{sec2.1}), and $\mu$ is a bounded Borel measure on $E$. 
The class of Dirichlet operators   is quite wide. In the important  case of $E= \BR^d$,  it includes densely defined operators $A$  of the form
\begin{equation}
\label{eq1.1abab}
Au(x)=\sum_{i,j=1}^d(a_{i,j}(x)u_{x_i})_{x_j}+P.V. \int_{\BR^d\setminus\{0\}} (u(x)-u(y))N(x,dy)+u(x)c(x),
\end{equation}
with   an elliptic matrix $a$, symmetric kernel  $N$, and $c\ge 0$ (see \cite[Theorem 3.2.3]{FOT}). 
%(if we replace $a,j,c$ by Radon measures, we get full characterization
%of self-adjoint Dirichlet  operators $A$ on $\BR^d$ with $C^1_c(\BR^d)\subset D(\sqrt {-A})$, see \cite[Theorem 3.2.3]{FOT}). In particular, the class of operators we consider here includes local operators (the model example is the uniformly elliptic divergence form operator) as well as non-local operators (the model example is the fractional Laplacian $\Delta^{\alpha/2}$ with $\alpha\in (0,2)$).

Since the seventies the Schr\"odinger operators with singular potentials have attracted growing interest in the literature (see e.g. \cite{AEG,AS,ABR,Eposito} and references therein).
The very important, among others, class of singular potentials appearing in applications, including, as particular case, Coulomb potentials
and harmonic potentials, is the family of  repulsive potentials  of the form  
%The class of smooth measure perturbations (depending on $A$) is very wide and  covers very important in applications Schr\"odinger operators.
%For instance, for  the fractional Laplacian $A=\Delta^{\alpha/2}$, $\alpha\in (0,2)$), the class of smooth measures includes 
\begin{equation}
\label{eq1.1a}
\nu_1(dx)=\sum_{j=1}^{N}\frac{c_j}{|x-x_j|^{\beta_j}}\,dx, \quad \nu_2(dx)=\frac{c_1}{\delta_D^{\beta_1}(x)}\,dx
\end{equation}
with any $c_j\ge 0$, $\beta_j\in\BR$, $x_j\in\BR^d$, $j=1,\dots,N$. Here and in what follows $\delta_D(x)=\mbox{dist}(x,\partial D)$. Another wide class, disjoint from  the one above, of singular potentials is the class  of {\em generalized  potentials},  i.e.  Borel measures $\nu$ 
concentrated on some  $m$-measure zero  set  $N\subset E$, e.g.   in case $A$ is the  fractional Laplacian $\Delta^{\alpha/2}$, i.e. the operator of the form \eqref{eq1.1abab}
with $a_{i,j}\equiv c\equiv 0$ and $N(x,dy)=|x-y|^{-d-\alpha}$ for some $\alpha\in (0,2)$,  any  $\sigma$-finite positive Borel measure $\nu$ satisfying
\[
\nu(dx)\ll \mathcal H^\lambda
\]
for some $\lambda\in (d-\alpha,d)$, where $\HH^\lambda$ is the Hausdorff measure of order $\lambda$, falls within the class of generalized potentials. 
These types of  potentials have  been considered in variety of models in  nuclear physics, solid-state physics and quantum field theory (see e.g. \cite{ABR}). 
The class $S_A$ of smooth measures  includes, as  particular cases, the above mentioned types of potentials. Although the class $S_A$ depends on $A$,
we always have the inclusion
\[
L^{1,+}_{loc}(E;m)\subset S_A.
\]
%This class is a far reaching extensions of the set of measures having locally integrable density with respect to $m$.
Note that smooth measure need not be  a Radon measure. In fact, it can be a  nowhere Radon measure. As an example of such a measure can serve $\nu_1$ defined by (\ref{eq1.1a}) with $N=\infty$ and suitably chosen  $\{c_j\}$ , $\{\beta_j\}$ and $\{x_j\}$ (see \cite{AM1}).

%there is a wide class of Dirichlet operators
%for which one can formulate  properly the definition of solution to (\ref{eq1.1}) and get existence and uniqueness result for (\ref{eq1.1})
%as well as approximation  of its solution by  the classical one.

%The last class of operators is recently intensively
%studied in the literature because of its  application in many
%physical and biological models.

One can look at  (\ref{eq1.1}) from two different perspectives.  In
the first one, we regard   (\ref{eq1.1}) as equation of the form
\begin{equation}
\label{eq1.2}
-A_\nu v=\mu,
\end{equation}
where $-A_\nu$ is a non-negative self-adjoint operator on $L^2(E;m)$
being the perturbation of $-A$ by the smooth measure potential
$\nu$, that is $-A_{\nu}=-A+\nu$.  In the second one, we regard
(\ref{eq1.1}) as the equation
\begin{equation}
\label{eq1.3}
-A u=-u\cdot\nu+\mu
\end{equation}
with absorption term on the right-hand side. The difference
between (\ref{eq1.2}) and (\ref{eq1.3})  is very subtle and
appears only in the case where the concentrated part $\mu_c$ of the
measure $\mu$, i.e. the singular part of $\mu$ with respect to
the capacity associated with $A$ is non-trivial.
The main goal of the paper is to study (\ref{eq1.1}) from the above
two perspectives. First, we   provide definitions of solutions to
(\ref{eq1.2}) and (\ref{eq1.3}). The problem of proper definitions
of solutions is rather delicate and requires using
some deep results  from the  potential theory.
We then give some  necessary and
sufficient condition for the existence of solutions
to (\ref{eq1.2}) and  to (\ref{eq1.3}), and we compare
the two approaches to (\ref{eq1.1}).
Finally, we give some results on regularity
and stability of solutions.

%The goal of the paper is to give the
%definitions of solution to (\ref{eq1.1}) from perspective of
%(\ref{eq1.2}) and (\ref{eq1.3}), and give  a necessary and
%sufficient condition for existence of (\ref{eq1.1}) using the
%approach (\ref{eq1.2}) and (\ref{eq1.3}). We also compare the two
%mentioned approaches and analyze properties of solutions.

% lies in the fact the we consider equations on different spaces; $L^2(E;\nu)$ and $L^2(E;m)$, respectively, and in the second case we have an absorption term on the right-hand side which forces additional integrability on the term $u\cdot\nu$
%(let us note that  smooth measure $\nu$ may be nowhere Radon).

In the paper, a solution $v$ to (\ref{eq1.2}) will be called a {\em duality solution} to (\ref{eq1.1}),
and a solution $u$ to (\ref{eq1.3}) will be called a {\em strong duality solution} to (\ref{eq1.1}).
Heuristically,
\begin{equation}
\label{eq1.s3}
v=R^{\nu}\mu,\qquad u=R(-u\cdot\nu+\mu),
\end{equation}
where
\[
R^\nu=(-A_\nu)^{-1},\qquad R=(-A)^{-1}.
\]
Note that both operators $R^\nu$ and $R$
are well-defined on $L^{\infty,+}(E;m)\cap L^2(E;m)$ by
\[
R^\nu\eta=\esssup_{\alpha>0}R^\nu_\alpha\eta,\qquad R\eta=\esssup_{\alpha>0} R_\alpha\eta,\quad \eta\in L^{\infty,+}(E;m)\cap L^2(E;m),
\]
where $(R^{\nu}_\alpha)_{\alpha>0}$, $(R_\alpha)_{\alpha>0}$ are the resolvents of $A_{\nu}$ and $A$, respectively (see, e.g., \cite{AM,Stollmann}).
Since the operators $R^\nu$ and  $R$ are linear and positive definite, we may extend them to $L^{\infty,+}(E;m)$ (with possibly infinite values).
For these extensions, we have $R^{\nu}\eta\le R\eta, m$-a.e. $\eta\in L^{\infty,+}(E;m)$.
Since the  operators $R^{\nu}, R$ are not defined on the space of measures, the idea is to understand (\ref{eq1.s3}) in the duality sense, i.e. we require that
\begin{equation}
\label{eq1.s4}
\int_E v\eta\,dm=\int_E R^\nu\eta \,d\mu,\quad\quad\quad \int_E u\eta=-\int_E uR \eta \,d\nu+\int_E R\eta\,d\mu
\end{equation}
for every $\eta\in L^\infty(E;m)$ such that $R|\eta|$ is bounded $m$-a.e. In the second equation, we additionally require that  $u\in L^1(E;\nu)$.
Although this idea is simple and natural, its implementation is complicated
by the fact that  $\mu, \nu$ are measures, and what is more, $\mu$ is an arbitrary bounded Borel measure. For this reason (\ref{eq1.s4}) is meaningful only if the operators $R^{\nu}$ and $R$ are defined
pointwise, i.e.
the functions $R^{\nu}\eta$ and $R\eta$ are well defined at every point of $E$ for every positive $\eta\in\BB(E)$.
We  can define $R$ pointwise  by using the  Green function $G$ for $-A$. Namely, we put
\begin{equation}
\label{eq1.s46s}
R\eta(x)=\int_E G(x,y)\,\eta(dy),\quad x\in E.
\end{equation}
Unfortunately, in general, there is no Green function for $-A_\nu$. One of the results of the paper consists in finding
a natural pointwise meaning for $R^\nu$. We propose such a version and denote by $\check R^\nu$. In the case of uniformly elliptic divergence form operator
\begin{equation}
\label{eq1.7} A= \sum^d_{i,j=1}\frac{\partial}{\partial
x_j}\Big(a_{ij}\frac{\partial}{\partial x_i}\Big)
\end{equation}
Malusa and Orsina \cite{MO} used the notion of Lebesgue's points to  define  the following version of the resolvent $R^\nu$:
\begin{equation}
\label{eq4.2ii}
\check R^\nu\eta(x)=\lim_{r\rightarrow 0^+} \frac{1}{|B(x,r)|}\int_{B(x,r)}R^\nu\eta(y)\,m(dy).
\end{equation}
Unfortunately,  this recipe for pointwise version of $R^{\nu}$ can be used only for a subclass of  operators whose harmonic functions are characterized  by the mean value property (or are  comparable, via Green function, with operators  for which the mean value characterization of harmonic functions holds; see \cite[Lemma 4.10]{MO}).
We propose a completely new approach based on
the probabilistic potential theory. Our approach considers that the key role
is played by the set
%to give a proper pointwise meaning for operator $\check R^\nu$ and  in understanding behaviour of solutions to (\ref{eq1.2}) and (\ref{eq1.3}) is the following set associated with smooth measure $\nu$:
\[
E_\nu=\{x\in E: \int_{V_x}G(x,y)\,\nu(dy)<\infty\mbox{ for some finely open neighborhood $V_x$ of } x\}
\]
considered  in the case of equations with operator (\ref{eq1.7}) in \cite{BDM,Sturm}. Write $N_\nu:=E\setminus E_\nu$ and recall
that the fine topology is the smallest topology under
which all excessive functions are continuous. Let $\BX$
%\[
%\mathbb X=\big((X_t)_{t\ge0},(P_x)_{x\in E\cup \{\Delta\}},\mathbb F=(\FF_t)_{t\ge0}, %(\theta_t)_{t\ge 0}\big)
%\]
be a Hunt process with life time $\zeta$ associated with the operator $A$.
The first  result of the paper (see Section 3), although rather technical, plays a pivotal rule in our approach. It  states that for every positive smooth measure $\nu$ there exists a   unique positive continuous additive functional $A^\nu$ (PCAF) of $\mathbb X$  with  exceptional set $N=N_\nu$  such that for all $x\in E_\nu$ and $\eta\in\BB^+(E)$,
\begin{equation}
\label{eq4.1ii}
E_x\int_0^\zeta\eta(X_r)\,dA^\nu_r=\int_E G(x,y)\eta(y)\,\nu(dy),
\end{equation}
for every $x\in N_\nu$,
\begin{equation}
\label{eq4.1ii.90}
P_x(A^\nu_t=\infty,\, t>0)=1,
\end{equation}
and $\psi_{A^\nu}$ defined as
\[
\psi_{A^\nu}(x)=E_x \int_0^\zeta e^{-t} e^{-A^\nu_t}\,dt, \quad x\in E,
\]
is finely continuous on $E$. Moreover, we prove that $N_\nu$ is the minimal exceptional set, in the sense that, if there exists a PCAF $A$ of $\mathbb X$ with exceptional set $N\subset N_\nu$ such that (\ref{eq4.1ii}) holds and
$\psi_A(x)=E_x \int_0^\zeta e^{-t} e^{-A_t}\,dt$ is finely continuous on $E$, then $N=N_\nu$ and
$P_x(A_t=A^\nu_t,\, t>0)=1,\, x\in E$. The above result was proved by Baxter, Dal Maso and Mosco \cite{BDM} in the case of Brownian motion (see also \cite{Sturm}). Although this result is purely probabilistic in nature, it plays key role in defining duality solutions to (\ref{eq1.1}).  We put
\begin{equation}
\label{eq4.3ii}
\check R^\nu\eta(x)=E_x\int_0^\zeta e^{-A^\nu_t} \eta(X_t)\,dt,\quad x\in E,
\end{equation}
and  show that this formula agrees with (\ref{eq4.2ii}) in the case where $A$ is defined by  (\ref{eq1.7}).  With this notion in hand, in Section \ref{sec4.1} we introduce the definition of a duality solution to (\ref{eq1.1}) by using the first formula in (\ref{eq1.s4}) with $R^\nu\eta$ replaced by $\check R^\nu\eta$. We then show that for every bounded Borel measure $\mu$ on $E$ there exists a unique duality solution to (\ref{eq1.1}).

It is worth noting here, that the  above-mentioned  approach to (\ref{eq1.1}) goes  back
to Stampacchia \cite{Stampacchia}, where equations with  measure
data and  operator (\ref{eq1.7}) defined on a bounded regular
domain $D\subset \BR^d$ are considered.  In \cite{Stampacchia},
the potential measure $\nu$ is of the form $\nu(dx)=V(x)\,dx$, where
$V\in L^p(D;m)$ with $p>d/2$. Under this assumption, there exists the
Green function for $-A_\nu$, so $\check R^\nu$ can be defined
through  its Green function.

As we mentioned earlier formula (\ref{eq4.3ii}) gives a natural  pointwise
meaning for the resolvent $R^\nu$.   In Section \ref{sec4.2}, we spend some
time to explain why we use here the term ``natural". First, we
show that if $\nu$ is a positive smooth measure such that there
exists the Green function $G^\nu$ for the operator $-A_\nu$
($\nu$ is then called a {\em strictly smooth measure}), then
\[
\check R^\nu \eta(x) =\int_E G^\nu(x,y)\eta(y)\,m(dy),\quad x\in
E.
\]
Moreover, for every sequence  $\{\nu_n\}$ of  positive strictly
smooth measures such that $\nu_n\nearrow \nu$,
\[
\check R^\nu\eta(x)=\lim_{n\rightarrow \infty}
\int_E G^{\nu_n}(x,y)\eta(y)\,m(dy),\quad x\in E.
\]
Then we show that if  $\mu$ is additionally continuous functional
on the extended domain $D_e(\EE)$, i.e. $\mu\in D'_e(\EE)$ (with
the inner product $\EE(\cdot,\cdot)$), then the duality solution $v$
to (\ref{eq1.1})  is the  unique minimizer of the energy
\[
E(\eta)=\frac{1}{2}\EE(\eta,\eta)+\frac12\int_E|\tilde \eta|^2\,d\nu
-\langle\mu, \eta\rangle_{D'_e(\EE), D_e(\EE)},\quad \eta\in
D_e(\EE)\cap L^2(E;\nu).
\]
Here $\tilde\eta$ stands for the quasi-continuous $m$-version of
$\eta$. In other words, $v$ is a variational solution to
(\ref{eq1.1}). Moreover, we show that for every bounded Borel
measure $\mu$, there exists a sequence $\{\mu_n\}$ of bounded
Borel measures in $D_e'(\EE)$ such that $\mu_n\rightarrow \mu$
narrowly and $v_n\rightarrow v$, where $v_n$ is a
variational solution to (\ref{eq1.1}) with $\mu$ replaced by
$\mu_n$. This stability property of $v$ is   sometimes used in the
literature as the definition of the so-called {\em SOLA solution} (see e.g. \cite{GV}).

Let us note here that variational approach to (\ref{eq1.1}) with
the operator (\ref{eq1.7}) on a bounded regular domain $D\subset
\BR^d$ and $\mu\in H^{-1}(D)$ was applied in Dal Maso and Mosco
\cite{DM1,DM} in the context of the so called relaxed Dirichlet
problem. In \cite[Example 3.10]{DM1} it is observed that, in
general, a variational solution to (\ref{eq1.1}) is not a
distributional solution to (\ref{eq1.1}). It is also worth
mentioning that in \cite{DM1,DM} the authors considered even more general class of
perturbations $\nu$ which do not satisfy quasi-finiteness condition (see condition (b) in Section \ref{sec2.1}).

In Section 4.3, we  prove basic regularity properties of a duality
solution $v$ to (\ref{eq1.1}).  We show that $v$ possesses an
$m$-version $\tilde v$ which is quasi-continuous and  $\tilde v\in
L^1(E;\nu)$.
% and
%\begin{equation}
%\label{eq4.4ii}
%\check v(x)=0,\quad x\in N_\nu.
%\end{equation}
Moreover,  we show that for every $k\ge 0$, $T_k(v):=
\min\{k,\max\{v,-k\}\}\in D_e(\EE)$ and
\[
\EE(T_k(v),T_k(v))\le k\|\mu\|_{TV}.
\]

Section 5 is devoted to strong duality solutions  to
(\ref{eq1.1}). By a solution we mean a quasi-continuous function
$u$ on $E$ such that  $u\in L^1(E;\nu)$ and the second equation in
(\ref{eq1.s4}) holds with $R\eta$ defined pointwise by
(\ref{eq1.s46s}). To understand the subtle difference between the
notion of duality and strong duality solution to (\ref{eq1.1}),
we have to take a closer look at the  formulations of both
definitions (see (\ref{eq1.s4})). Observe that in the case of
duality solutions, we consider some class of test functions
included in the range of the operator $\check R^\nu$, and in the
case of strong duality solution, we consider possibly wider class
of test functions included in
the range of operator $R$ (see Proposition \ref{eq4.prop.s3}).
In the first case, by
(\ref{eq4.1ii.90}) and (\ref{eq4.3ii}), each test function equals
zero on the set $N_\nu$. Hence,  for every $x\in N_\nu$, the
function $v\equiv 0$ is a duality solution to
\[
-Av+v\cdot\nu=\delta_{\{x\}}.
\]
However, this function is not a strong duality solution to the
above equation as in such a case, by the
definition, we would have
\[
0=\int_E R\eta\,d\delta_{\{x\}}=R\eta(x)=\int_E G(x,y)\eta(y)\,m(dy).
\]
This implies that $G(x,\cdot)=0$, which contradicts the
definition of the Green function. We prove that there exists a
strong duality solution $u$ to (\ref{eq1.1}) if and only if
\[
|\mu_c|(N_\nu)=0,
\]
and in this case $u$ is also a duality solution to (\ref{eq1.1}).
Recently, this result was proved by Orsina and Ponce \cite{OP} in
the case where $A=\Delta_{|D}$ and $\nu(dx)=V(x)\,dx$, and by G\'omez-Castro and   Vazquez  \cite{GV} 
in case $A=\Delta^{\alpha/2}$, $\nu(dx)=V(x)\,dx$ with $V$ belonging to the class (\ref{eq1.pscvr1}) defined below.

An existence result for strong duality solutions to (\ref{eq1.1})
is a direct consequence of the following, interesting in its own
right, result which we prove in Section 5. It states that if $v$
is a duality solution to (\ref{eq1.1}), then it is a strong
duality solution to
\[
-Av+v\cdot\nu=\mu_{\lfloor E_\nu}.
\]
We have already mentioned that  $v$ is a limit of  variational solutions to the
Schr\"odinger equations
\begin{equation}
\label{eq1.rmo1}
-Av_n+v_n\nu=\mu_n
\end{equation}
with more regular than $\mu$ measures $\mu_n$ approximating $\mu$
in the narrow topology.
This means that when passing to the limit  in (\ref{eq1.rmo1}) some reduction of the measure $\mu$ occurs.
This  phenomenon is somewhat reminiscent of the phenomena occurring
in the theory of reduced measures introduced by Brezis,
Marcus and  Ponce \cite{BMP,BMP1} for the Dirichlet Laplacian
and next generalized by Klimsiak  \cite{K:CVPDE} to a wide class
of Dirichlet operators. In our context, the  measure
$\mu_{\lfloor E_\nu}$ may be considered as a reduced measure for
$\mu$.  The same reduction takes place when we approximate
monotonically the measure $\nu$ (see Proposition \ref{prop4.s1s}).
This result is a far-reaching  generalization of  \cite[Theorem 8.2]{GV}.

In Section \ref{sec6} we briefly comment on the easy extension of our existence results
to the weighted  measure spaces $\MM_\rho$ consisting of Borel measures $\mu$ with
\[
\int_E\rho\,d|\mu|<\infty,
\]
where $\rho$ is a strictly positive excessive function, e.g. the principle eigenfunction for $-A$.
In particular, we cover the class $\MM_{\delta_D^{\alpha/2}}$  considered in \cite{DGV} for $C^{1,1}$ open domain $D\subset \BR^d$ and 
$A=(\Delta^{\alpha/2})_{|D}$, i.e. Dirichlet fractional Laplacian with zero exterior condition. It is well know that for $C^{1,1}$ domains
$\delta_D^{\alpha/2}$ is comparable to the principle eigenfunction of $(\Delta^{\alpha/2})_{|D}$. We close the section with comments
on the notion of  renormalized solutions and its relation to the notion of duality solutions.

Self-adjoint Schr\"odinger operators with smooth measure potentials (also called generalized Schr\"odinger operators) and their applications  to quantum theory
were intensively studied  in  the late '70s and '80s of the last century
by using  methods of Dirichlet forms, probabilistic potential theory and harmonic spaces, see  the papers by Albeverio, Ma and R\"ockner \cite{ABR,AM,AM1} and the paper by Boukricha, Hansen and Hueber \cite{BHH} for a nice account of results in this direction. At the same time, Baxter, Dal Maso and Mosco \cite{BDM,DM1,DM} studied equations of the form (\ref{eq1.1}) with the classical Laplacian and $\mu\in H^{-1}(D)$ in the context of the so-called relaxed Dirichlet problem with even more general class of potentials, which
do not satisfy the quasi-finiteness condition which is required in the definition of a smooth measure. 
In '90s and 2000s,  equation (\ref{eq1.1}) with smooth both $\nu$ and $\mu$ was studied by Getoor \cite{G1,G2,G3}  and Beznea and Boboc \cite{BB}
with more general class of operators $A$ generated by  right Markov semigroups.

Recently, Orsina and Ponce \cite{OP} and Ponce and Wilmet \cite{PW}
considered Schr\"odinger equations of the form (\ref{eq1.1}) with $A$ being the classical Laplacian, 
$\nu(dx)=V(x)\,dx$ for some positive Borel measurable $V$, and $\mu$ being a general bounded Borel measure. 
As to the non-local Schr\"odinger operators,  Diaz, G\'omez-Castro and   Vazquez  proved in \cite{DGV} existence results
for (\ref{eq1.1}) with $A=\Delta^{\alpha/2}$, $\nu(dx)=V(x)\,dx,\, V\in L^{1,+}_{loc}(D;m)$, and $\mu\in L^1(D; \delta^\alpha_D\cdot m)$.
Next,  G\'omez-Castro and   Vazquez in \cite{GV}  generalized results of \cite{DGV} by   considering   more relaxed class of potentials:
\begin{equation}
\label{eq1.pscvr1}
L^{1,+}_{loc}(D)+ \big\{g\in \BB^+(D): \exists_{\mbox{finite}\, Y\subset D};\, g\in L^\infty(D\setminus \bigcup_{x\in Y} B(x,r))\,\,\mbox{for any}\,\, r>0\big\}.
\end{equation}
They also considered  bounded Borel measures $\mu$ in (\ref{eq1.1}). It is an elementary  calculation that  for  a potential $V$ belonging to \eqref{eq1.pscvr1},
we have  $V(x)\,dx\in S_A$.
However, the general case, which we are interested in the paper,  i.e. equations of the form (\ref{eq1.1}) with smooth measure $\nu$ and general
(not necessarily smooth) measure $\mu$ were until now considered only  by Malusa and Orsina  \cite{MO} in the case where $A$ is a uniformly elliptic divergence form operator,
and under additional condition that $\mu$ has compact support.

The  main goal of the present paper is twofold. First, we give a necessary and
sufficient condition for the existence of a solution to the
 Schr\"odinger equation \eqref{eq1.1} in our general framework, thus we recover and give a far-reaching generalization of  the existence results of \cite{DGV,GV,MO,OP,PW}.
 Secondly, we provide a unified method, based on the probabilistic potential theory, to investigate Schr\"odinger equations of type   (\ref{eq1.1}) with general measure data.
 Of course, probabilistic methods based on Feynman-Kac formula have been applied successfully to Schr\"odinger equations in the past (see e.g. \cite{AS}),
 however in vast majority of the papers $\mu$ in (\ref{eq1.1}) was assumed to be a smooth measure. 
 This is due to the fact that only smooth measures are in {\em Revuz duality} with {\em positive additive functionals}, and this duality  is crucial for the Feynman-Kac representation
 formulas.

%We see that once again the notion of the set $N_\nu$ plays
%a crucial role in  the existence result to Schr\"odinger equations.

\section{Preliminary results}
\label{sec2}

In this section, we  make standing assumptions on the Dirichlet
operator and the associated Dirichlet form considered in the
paper. For the convenience of the reader, we also recall some basic
facts from the potential theory and the probabilistic potential
theory.

\subsection{Dirichlet forms and potential theory}
\label{sec2.1}

In the whole paper, we assume that $(A,D(A))$ is a non-positive
self-adjoint
operator on $L^2(E;m)$ generating a strongly continuous Markov
semigroup of contractions $(T_t)_{t\ge 0}$ on $L^2(E;m)$. It is well known (see
\cite[Section 1]{FOT}) that there exists a unique symmetric
Dirichlet form $(\EE, D(\EE))$ on $L^2(E;m)$ such that
$D(A)\subset D(\EE)$ and
\[
\EE(u,v)=(-Au,v),\quad u\in D(A), v\in D(\EE).
\]
We denote by $(J_\alpha)_{\alpha>0}$ the resolvent generated by
$A$. We assume that $(\EE,D(\EE))$ is {\em transient} and {\em regular}, i.e.
there exists a strictly positive bounded function $g$ on $E$ such
that
\[
\int_E|u|g\,dm\le \sqrt{\EE(u,u)},\quad u\in D(\EE),
\]
and $D(\EE)\cap C_c(E)$ is dense in $C_c(E)$ and in $D(\EE)$  with
natural  topologies on these spaces. Since $\EE$ is transient,
there exists an extension $D_e(\EE)\subset L^1(E;\,g\cdot m)$ of
the domain $D(\EE)$ such that the pair $(\EE,D_e(\EE))$ is a
Hilbert space. For an open $U\subset E$ we set
\[
\mbox{Cap}(U)=\inf\{\EE(u,u): u\ge\mathbf{1}_U\mbox{ $m$-a.e.},\,
u\in D(\EE)\}.
\]
and then, for arbitrary $B\subset E$, we set $\mbox{Cap}(B)=\inf
\mbox{Cap}(U)$, where the infimum is taken over all open subsets
of $E$ such that $B\subset U$. We say that a property $P$ holds
$(\EE)$-q.e. if it holds except a set of capacity $\mbox{Cap}$ zero.

We say that a function $u$ on $E$ is $(\EE)${\em -quasi-continuous} if for every
$\varepsilon>0$ there exists a closed set $F_\varepsilon\subset E$
such that $\mbox{Cap}(E\setminus F_\varepsilon)\le\varepsilon$ and
$u_{|F_\varepsilon}$ is continuous. By \cite[Theorem 2.1.3]{FOT},
each function $u\in D_e(\EE)$ admits a quasi-continuous
$m$-version. In the sequel, for $u\in D_e(\EE)$, we denote by
$\tilde u$ its quasi-continuous $m$-version. By \cite[Lemma 2.1.4]{FOT} if for given $(\EE)$-quasi-continuous functions $u,v$
we have  $u=v$, $m$-a.e., then $u=v$ $(\EE)$-q.e.

We say that a positive Borel measure $\nu$ on $E$ is $(\EE)$-{\em smooth}  if
\begin{enumerate}[(a)]
\item $\nu$ is absolutely
continuous with respect to the capacity $\mbox{Cap}$ generated by the Dirichlet form $(\EE,D(\EE))$,

\item there exists an increasing sequence $\{F_n\}$ of closed subsets of $E$ such that for any compact $K\subset E$, 
\[
\nu(F_n)<\infty,\quad n\ge 1,\qquad\lim_{n\rightarrow\infty}\mbox{Cap}(K\setminus F_n)\rightarrow 0.
\]
\end{enumerate}
A signed Borel measure $\nu$ on $E$ is smooth if its variation $|\nu|$ is smooth.
A sequence $\{F_n\}$
satisfying condition (b) of the above definition  is called a {\em generalized nest}.

Let $\BB(E)$ (resp. $\BB^+(E)$) denote the set of all Borel (resp.
non-negative Borel) measurable functions on $E$. We adopt the
following notation: for a positive Borel measure $\mu$ on $E$ and
$f\in \BB^+(E)$ we set
\[
\langle\mu, f\rangle=\int_E f\,d\mu,
\]
and  we denote by $f\cdot \mu$ the  Borel measure on $E$ such that
\[
\langle f\cdot\mu,\eta\rangle=\langle \mu,f\eta\rangle,\quad \eta\in \BB^+(E).
\]
If $\mu=m$, we write $\langle f,\eta\rangle=\langle f\cdot m,\eta\rangle$.
By $\MM_1$ we denote the set of Borel measures $\mu$ on $E$ for which $\langle |\mu|,1\rangle<\infty$.
We say that $\mu$ is a bounded Borel measure if $\mu\in\MM_1$.
Let $\nu$ be a positive smooth
measure. We set
\[
D(\EE^\nu)= D(\EE)\cap L^2(E;\nu),\qquad
\EE^\nu(u,v)=\EE(u,v)+\langle \tilde u\cdot\nu,\tilde
v\rangle,\quad u,v\in D(\EE^\nu).
\]
By \cite[Theorem 4.6]{MR}, $(\EE^\nu,D(\EE^\nu))$ is a
quasi-regular symmetric Dirichlet form on $L^2(E;\nu)$. By
\cite[Corollary 2.10]{MR}, there exists a unique non-positive self-adjoint
operator $(A^\nu, D(A^\nu))$  such that $D(A^\nu)\subset
D(\EE^\nu)$ and
\[
\EE^\nu(u,v)=(-A^\nu u,v),\quad u\in D(A^\nu), v\in D(\EE^{\nu}).
\]
We denote by $(J^\nu_{\alpha})_{\alpha>0}$ the resolvent generated
by $-A^\nu$, and by $(T^\nu_t)_{t\ge 0}$ the strongly continuous
Markov semigroup of contractions generated by $-A^\nu$.

\subsection{Probabilistic potential theory}
\label{sec2.2}

Let $\Delta$ be a one-point compactification of $E$, in case $E$ is not compact, or  an isolated point, in case $E$ is compact.  
Let $\cal D$ denote  the set of all 
functions $\omega: [0,\infty)\to E\cup\{\Delta\}$, that are right continuous and
possess the left limits for all $t\ge0$ (c\'adl\'ags), and  have the property that if $\omega(t)=\Delta$,
then $\omega(s)=\Delta,\, s\ge t$. We equip $\mathcal D$ with
the Skorokhod topology, see e.g. Section 12 of \cite{bil}. 
Define the {\em canonical process}: $X:\mathcal D\to\mathcal D$, $X_t(\omega):=\omega(t)$, $\omega \in \mathcal D$,
{\em shift operator}: $\theta_s:\mathcal D\to\mathcal D$, $\theta_s(\omega)(t)=\omega(t+s)$,  and   {\em life time}: $\zeta: \mathcal D\to [0,\infty]$,   $\zeta(\omega):=\inf\{t>0: X_t(\omega)=\Delta\}$.
By \cite[Theorem 7.2.1]{FOT}, there exists a Hunt process
\[
\mathbb X=((P_x)_{x\in E\cup \{\Delta\}},\,
\mathbb{F}=(\FF_t)_{t\ge0})
\]
- here $P_x$ is a probability measure on $\mathcal D$ for fixed $x\in E$ ($P_\Delta=\delta_{\{\Delta\}})$,
$\mathbb F$ is a filtration (non-decreasing sequence of $\sigma$-fields) -  such that for all $t\ge 0$ and  $f\in
\BB(E)\cap L^2(E;m)$,
\[
T_tf(x)=E_xf(X_t):=\int_{\mathcal D} f(X_t(\omega))\,dP_x(\omega)\quad m\mbox{-a.e. }x\in E,
\]
with the convention that $f(\Delta)=0$. For $f\in\BB^+(E)$, we put
\[
P_t f(x)=E_xf(X_t),\,\,t\ge0, \qquad R_\alpha f(x)=E_x\int_0^\zeta
e^{-\alpha t}f(X_t)\,dt,\,\,\alpha\ge0,
\]
and $R= R_0$. Recall  that  $f\in\BB^+(E)$ is an {\em excessive function} (with respect to $\mathbb X$) if
\[
\sup_{\alpha >0}\alpha R_\alpha f(x)=f(x),\quad x\in E.
\]
If  in the above condition $R_{\alpha}$ is replaced by $R_{\alpha+\beta}$ for some $\beta>0$ , then $f$
is called $\beta$-excessive (with respect to $\mathbb X$). 
We assume that $\mathbb X$ satisfies the {\em absolute
continuity condition}, i.e.  for any $f\in\BB^+_b(E)$,
\[
P_tf(x)=0,\, t>0, \, x\in E\quad\mbox{ whenever}\quad \int_Ef\,dm=0.
\]
By \cite[Lemma 4.2.4]{FOT} (by transiency of $\mathbb X$, \cite[Lemma 4.2.4]{FOT} holds for $\alpha=0$, too)
for any $\alpha\ge 0$ there exists $r_\alpha\in\BB^+(E\times E)$ such that $r_\alpha(x,\cdot), r_\alpha(\cdot,y)$
is $\alpha$-excessive for fixed $x,y\in E$, and 
\[
R_\alpha f(x)=\int_Ef(y)r_\alpha(x,y)\,m(dy),\quad x\in E,\, f\in\BB^+(E).
\]
%there exists a non-negative  Borel
%function $p:(0,\infty)\times E\times E\rightarrow \BR$ such that
%for every  $f\in \BB^+(E)$,
%\[
%P_tf(x)=\int_E p(t,x,y)f(y)\,m(dy),\quad x\in E,\, t>0.
%\]
We set $G_\alpha= r_\alpha$, and $G=G_0$. We call $G$ the {\em Green function} for $-A$ (or $\mathbb X$).
For a given positive Borel measure $\mu$ on $E$, we set
\begin{equation}
\label{eq.efmd12}
R_\alpha \mu(x)=\int_E G_\alpha(x,y)\,\mu(dy),\quad x\in E.
\end{equation}
Observe that $P_t(R_\alpha\mu)=\int_E P_tG_\alpha(\cdot,y)\,\mu(dy)$.
Therefore, since $G_\alpha(\cdot,y)$ is $\alpha$-excessive, for any positive $\mu$ we have that $R_\alpha\mu$ is $\alpha$-excessive as well.
A Borel measure $\mu$ on $E$ is called  {\em strictly smooth} (with respect to $\mathbb X$)  if it is
$(\EE)$-smooth and  there exists an increasing sequence $\{B_n\}$ of Borel
subsets of $E$ such that $\bigcup_{n\ge 1} B_n=E$ and
$R(\mathbf{1}_{B_n}\cdot|\mu|)$ is bounded for every $n\ge 1$.

From the definition of an excessive function it
follows directly that under the  absolute continuity condition for
$\mathbb X$, if $f\le g$ $m$-a.e for some excessive
functions $f,g$, then $f\le g$ on $E$. We will use this property
frequently in the paper without special mentioning.

We say that $A\subset E$ is {\em nearly Borel} (with respect to $\mathbb X$) if there exist $B_1,
B_2\in\BB(E)$ such that $B_1\subset A\subset B_2$ and for every
finite positive Borel measure $\mu$ on $E$,
\[
P_\mu(\exists\,{t\ge 0}\,\,\, X_t\in B_2\setminus B_1)=0,
\]
where $P_\mu(d\omega)=\int_E P_x(d\omega)\,\mu(dx)$. By
$\BB^n(E)$, we denote the class of all nearly Borel sets. It is
clear that $\BB(E)\subset \BB^n(E)$. For  $A\in \BB^n(E)$, we set
\[
\sigma_A=\inf\{t>0: X_t\in A\},\qquad \tau_A=\sigma_{E\setminus
A}\wedge\zeta.
\]
We say that a set $A\subset E$ is   {\em polar} (with respect to $\mathbb X$)  if there exists $B\in
\BB^n(E)$ such that $A\subset B$ and
\[
P_x(\sigma_B<\infty)=0,\quad x\in E.
\]
By \cite[Theorem 4.1.2, Theorem 4.2.1]{FOT}, $\mbox{Cap}(A)=0$ if
and only if $A$ is polar. A nearly Borel set $D$ is an  {\em absorbing set} (with respect to $\mathbb X$) if 
$P_x(\tau_D=\zeta)=1$, $x\in D$.
Observe that if $N$ is a polar set, then $E\setminus N$ is an absorbing set.

Let $\TT$ be the topology generated by the metric on $E$.  By
$\TT_f$ we denote the smallest topology on $E$ for which all
excessive functions (with respect to $\mathbb X$) are continuous. This topology is called in the
literature the {\em fine topology} (with respect to $\mathbb X$).  By \cite[Section II.4]{BG}, $\TT\subset
\TT_f$ and $A$ is a finely open set if and only if for every $x\in
A$ there exists $D\in\BB^n(E)$  such that $D\subset A$ and
\[
P_x(\tau_D>0)=1.
\]
In other words, starting from  $x\in A$,  the process $X$ spends some
nonzero  time  in $A$ until it exits $A$. Observe that
each polar set is finely closed. By \cite[Theorem II.4.8]{BG},
$f\in \BB^n(E)$ is finely continuous if and only if the process
$f(X)$ is right-continuous under the measure $P_x$ for every $x\in
E$.

Let $D\subset E$ be finely open. We denote by
\[
\mathbb X^D=((P^D_x)_{x\in D\cup \{\Delta\}},\,\mathbb F^D=(\FF^D_t)_{t\ge0})
\]
the part of the process $\mathbb X$ on $D$ (see \cite[Section
A.2]{FOT}). Its life time satisfies $\zeta=\tau_D$.  By \cite[Theorem
A.2.8]{FOT}, $\mathbb X^D$ is again a Hunt process and
\[
P^D_tf(x):=E^D_xf(X_t)=E_x f(X_t)\mathbf{1}_{\{t<\tau_D\}},\quad x\in D,
\]
where $E^D_x$ denotes the expectation with respect to the measure $P^D_x$. We also set
\begin{equation}
\label{eq3.0sd}
R^Df(x)= E_x\int_0^{\tau_D} f(X_t)\, dt=\int_0^\infty P_t^Df(x)\,dt,\quad x\in D.
\end{equation}
From this formula, we deduce that if $D$ is an  {\em absorbing set} (with respect to $\mathbb X$), i.e.
$P_x(\tau_D=\zeta)=1$, $x\in D$, then
\begin{equation}
\label{eq3.s1}
R^Df(x)=Rf(x),\quad x\in D.
\end{equation}
It is clear that $\mathbb X^D$ satisfies the absolute continuity condition. Therefore there exists a Green function $G^D$ for $\mathbb X^D$. From (\ref{eq3.s1}) it follows that if $D$ is an absorbing set, then
\begin{equation}
\label{eq3.s2}
G^D(x,y)= G(x,y),\quad x,y\in D.
\end{equation}

Observe that most of  the notions introduced in Section \ref{sec2.1}, \ref{sec2.2} depend on the form  $\EE$ or the process $\mathbb X$.
To underline this in all the foregoing definitions, where this dependence occurs, we  added $\EE$ or $\mathbb X$ in parenthesis.
In what follows all the notions  depending  on a form or a process  are  by default understood as the notions with respect to $\EE$ or $\mathbb X$ 
(we do not  indicate them in the sequel)
unless it is stated otherwise.

\section{PCAFs of $\mathbb X$ with minimal exceptional set}

The notion of a positive continuous additive functional (PCAF) of a
general Markov process was introduced by Revuz in \cite{Revuz}. This notion allowed the author to show
in  \cite{Revuz}  a duality between a subclass of smooth
measures  and a class of  increasing  processes with additivity
property. Unfortunately,  the subclass of smooth
measures considered by Revuz was too restrictive in many applications. It  did not even
cover the class of bounded smooth measures. To get the general
duality, covering the class of all smooth measures, Fukushima (see \cite{FOT}) and Silverstein (see
\cite{Si}) extended the notion of PCAF (see Definition
\ref{def.3.3.4} below). The crucial ingredient of the extended
definition is the notion of  so-called exceptional set $N$ which
depends on the particular PCAF. In the extended definition  the
desirable properties of the functional -  additivity, continuity
etc. -  are satisfied under the measure $P_x$ for $x\notin N$.
Thanks to the more general notion of PCAF one can  get a
one-to-one correspondence between the class of all  smooth
measures and PCAFs. Nowadays,  PCAFs in the sense of Revuz are
called strict PCAFs to distinguish them from PCAFs introduced by
Fukushima and Silverstein. Smooth measures associated with strict
PCAFs are called strictly smooth measures.

In the present section, we  show that the exceptional set $N$ for
a given PCAF can be chosen in a canonical way and that this choice
is in some sense minimal. In the special case when  $\mathbb X$
is a Brownian motion, our result follows from  the paper by
Baxter, Dal Maso and  Mosco \cite{BDM} (see also \cite{Sturm}).

In what follows, we say that  some property  holds a.s. if it holds $P_x$-a.s. for every $x\in E$.

\begin{definition}
\label{def.3.3.4}
We say that an $\mathbb F$-adapted process $A=(A_t)_{t\ge0}$ is a {\em positive continuous additive functional}
of $\mathbb X$ (PCAF) if there exists a polar set $N$ and $\Lambda\in\FF_\infty$ such that
\begin{enumerate}[(a)]
\item$P_x(\Lambda)=1,\,\,\, x\in E\setminus N$,
\item $P_x(A_0=0)$, $x\in E\setminus N$, and $A_t(\omega)\ge 0$, $t\ge 0$, $\omega\in \Lambda$,

\item $\theta_t(\Lambda)\subset \Lambda,\, t>0$, and for every $\omega\in \Lambda$,
$A_{t+s}(\omega)=A_t(\omega)+A_s(\theta_s\omega)$, $s,t\ge 0$,

\item $A_t(\omega)<\infty$, $t<\zeta(\omega)$, $\omega\in\Lambda$, and
$A_t(\omega)=A_{\zeta(\omega)},\,\,\,  t\ge \zeta(\omega),\, \omega\in\Lambda$,
\item $[0,\infty)\ni t\mapsto A_t (\omega)$ is continuous for every $\omega\in\Lambda$.
\end{enumerate}
\end{definition}

The set $N$ is called an {\em exceptional set for} $A$ and $\Lambda$ is called a {\em defining set for} $A$.
If $N=\emptyset $, then $A$ is called  a {\em strict PCAF} of $\mathbb X$. Notice  that $A$ is a PCAF of $\mathbb X$ if and only if $A$ is a strict PCAF of $\mathbb X^{E\setminus N}$.

The one-to-one correspondence between PCAFs $A$ of $\mathbb X$ and positive smooth measures $\nu$
is characterized by the relation
\[
\lim_{t\rightarrow 0^+}\frac{1}{t}E_m\int_0^t f(X_t)\,dA_t=\langle \nu,f\rangle,\quad f\in\BB^+(E).
\]
In the literature it is often called  the {\em Revuz duality}. If $A$ is a strict PCAF, then under our assumption of absolute continuity of $\mathbb X$, the above relation may be expressed equivalently as
\[
E_x\int_0^\zeta f(X_t)\,dA_t=\int_E G(x,y)f(y)\,\nu(dy),\quad x\in E,
\]
for all $f\in \BB^+(E)$.

By \cite[Theorem 5.1.7]{FOT}, for every positive smooth  measure
$\nu$ such that $R\nu$ is bounded, there exists a strict PCAF  $A$
of $\mathbb X$ in the Revuz duality with $\nu$. Let $\nu$ be a
positive smooth measure. By \cite[Theorem 2.2.4]{FOT}, there
exists a generalized nest  $\{F_n\}$ of closed subsets of $E$ such
that $R\nu_n\le n$, $n\ge1$,  and $\nu_n\nearrow \nu$, where
$\nu_n=\mathbf{1}_{F_n}\cdot\nu$ (the convergence follows from the
fact that $\{F_n\}$ is a generalized nest).

In what follows we adopt the following notation. For  given $\alpha\ge 0$, $f\in\BB^+(E)$, and non-negative $\mathbb F$-adapted  c\`adl\`ag process  $Y$,
we let 
\begin{equation}
\label{eq.dofcaf12}
\phi^{\alpha,f}_Y(x)=E_x\int_0^\zeta e^{-\alpha t}f(X_t)e^{-Y_t}\,dt,\quad x\in E.
\end{equation}
We also set $\phi^{\alpha}_Y:= \phi^{\alpha,1}_Y$.

\begin{proposition}
\label{prop4.1}
Let $\nu$ be a positive smooth measure on $E$ and let
\begin{align}
\label{eq3.1}
E_\nu=\{x\in E:\exists\, V_x & \mbox{-finely open neighborhood
of $x$}\nonumber \\
&\mbox{such that } \int_{V_x} G(x,y)\,\nu(dy)<\infty\}.
\end{align}
Let $\{\nu_n\}$ be a sequence of positive strictly smooth measures
such that $\nu_n\nearrow \nu$, and for $n\ge1$ let $A^n$ be a
strict PCAF of $\mathbb X$ in the Revuz correspondence with
$\nu_n$. Set $A_t:=\sup_{n\ge 1} A^n_t$, $t\ge 0$. Set $N_\nu:=E\setminus E_\nu$. Then
\begin{enumerate}[\rm(i)]
\item $\phi^{\alpha,f}_{A}$
is finely continuous for any $\alpha>0$ and $f\in\BB^+_b(E)$,
\item If $f\in\BB^+(E)$ and $Rf$ is finite, then $\phi^{0,f}_{A}$
is finely continuous.
\item The process
$A$ is a PCAF of $\BX$ with
the exceptional set $N_\nu$.

\item For all $x\in E_\nu$ and $\eta\in\BB^+(E)$,
\begin{equation}
\label{eq4.1} E_x\int_0^\zeta \eta(X_t)\,dA_t=\int_E
G(x,y)\eta(y)\,\nu(dy).
\end{equation}

\item For any $\alpha\ge 0$, $E_\nu=\{\phi^\alpha_A>0\}$.
\item For every $x\in N_\nu$, $P_x(A_t=\infty,\, t> 0)=1$.
\end{enumerate}
\end{proposition}
\begin{proof}
(i) Set 
\[
u^n_\alpha(x)=E_x\int_0^\zeta e^{-\alpha t}f(X_t)(1-e^{-A^n_t})\,dt,\quad x\in E.
\]
Clearly, $(u^n_\alpha)_{n\ge 1}$ is nondecreasing. Set $u_\alpha=\sup_{n\ge 1} u^n_\alpha$.
By the strong Markov property and additivity of $A^n$,
\[
e^{-\alpha t}P_t u^n_\alpha(x)=E_{x}\int_t^\zeta e^{-\alpha s}f(X_s)(1-e^{-(A^n_s-A^n_t)})\,ds,\quad x\in E.
\]
From this we  easily deduce that $u^n_\alpha$ is $\alpha$-excessive  for any $\alpha>0, n\ge 1$. 
Hence, by \cite[4.d), page 222]{DellacherieMeyer}, $u_\alpha$ is $\alpha$-excessive. 
Observe that
\[
\phi_{A^n}^{\alpha,f}(x)=u_\alpha^n(x)-R_\alpha f(x),\quad x\in E.
\]
Letting $n\rightarrow \infty$, we get
\[
\phi_{A}^{\alpha,f}(x)=u_\alpha(x)-R_\alpha f(x),\quad x\in E.
\]
Since $u_\alpha, R_\alpha f$ are finely continuous (as excessive functions) and finite on $E$, we conclude from the above equation that $\phi^{\alpha,f}_{A}$
is finely continuous for any $\alpha>0$. For (ii), we let $\alpha\to 0^+$  in the above equation.
We thus get 
\[
\phi_{A}^{0,f}(x)=u_0(x)-R f(x),\quad x\in E.
\]
Clearly, $u_0, Rf$ are  excessive functions, and hence finely-continuous.
Since $Rf$ is finite, $u_0-Rf$ is finely continuous.
For (iii), observe first that  $\phi^\alpha_A(x)>0$, $x\in E_\nu$. Indeed,
\begin{align}
\label{eq4.02}
\nonumber
\int_{V_x}G(x,y)\,\nu(dy)&= \lim_{n\rightarrow \infty}\int_E\mathbf{1}_{V_x}(y) G(x,y)\,\nu^n(dy)\\
&=
\lim_{n\rightarrow \infty}E_x\int_0^\zeta \mathbf{1}_{V_x}(X_t)\,dA^n_t
\ge \lim_{n\rightarrow \infty} E_xA^n_{\tau_{V_x}\wedge\zeta}= E_xA_{\tau_{V_x}\wedge\zeta}=E_xA_{\tau_{V_x}}.
\end{align}
Since $V_x$ is finely open, $P_x(\tau_{V_x}>0)=1$. From this and the definitions of $E_\nu$ and $\phi^\alpha_A$ we deduce that $\phi^\alpha_A$ is strictly positive on $E_\nu$.
By \cite[Lemma
5.1.5(ii)]{FOT}, we have $E_x\int_0^\zeta e^{-t}
\phi^1_{A^k}(X_t)\,dA^k_t\le 1$, $x\in E$, so $E_x\int_0^\zeta
e^{-t} \phi^1_{A^k}(X_t)\,dA^l_t\le 1$, $x\in E$ for $k\ge l$.
Letting $k\rightarrow \infty$, we obtain
\begin{equation}
\label{eq4.01}
E_x\int_0^\zeta e^{-t} \phi^1_A(X_t)\,dA^l_t\le 1,\quad x\in E,\quad l\ge 1.
\end{equation}
Since $\phi^1_A$ is finely continuous and strictly positive on
$E_\nu$, and $N_\nu$ is polar, we conclude from (\ref{eq4.01}) that
$A_t<\infty$, $t<\zeta$, $P_x$-a.s. for $x\in E_\nu$.  Hence, by
\cite[Lemma 1, page 182]{Meyer} (see also \cite[Lemma 1.1]{MW})
applied to the sequence $\{A^n\}$ regarded as a sequence of strict
PCAFs of $\mathbb X^{E_\nu}$, we get that $A$ is a strict PCAF of
$\mathbb X^{E_\nu}$, which is equivalent to the statement that $A$
is a PCAF  of $\mathbb X$ with the exceptional set $N_\nu$.  This
proves (iii). For $x\in E_\nu$, and $\eta\in\BB^+(E)$,  we have
\begin{align*}
E_x\int_0^\zeta \eta(X_t)\,dA_t&=\lim_{n\rightarrow \infty}
E_x\int_0^\zeta \eta(X_t)\,dA^n_t\\& =\lim_{n\rightarrow \infty}
\int_E\eta(y)G(x,y)\nu^n(dy)=\int_E\eta (y)G(x,y)\nu(dy).
\end{align*}
This completes the proof of  (iv).
By (\ref{eq4.02}), $E_\nu\subset \{\phi^\alpha_A>0\}$ for any $\alpha\ge 0$. Suppose that
$\phi^\alpha_A(x_0)>0$ for some $x_0\in E$. Then there exists $c>0$ such that $\phi^0_A(x_0)>c$.
By \cite[Corollary 1.3.6]{Oshima},  there exists a strictly positive bounded Borel measurable
function $g$ such that $Rg$ is bounded. Let $c_1=\sup_{x\in E}
Rg(x)$.  By the Lebesgue monotone convergence theorem 
\[
\phi^{1/n,g_n}_A(x)\nearrow \phi^{0}_A(x),\, x\in E,
\]
where $g_n=ng/(1+ng),\, n\ge 1$. Thus, there exists $n_0$ such that $\phi^{1/n_0,g_{n_0}}_A(x_0)>c$.
Write $V_{x_0}=\{\phi^{1/n_0,g_{n_0}}_A(x_0)>c\}$. Of course $x_0\in V_{x_0}$, and since $\phi^{1/n_0,g_{n_0}}_A$
is finely continuous (by (i)), $V_{x_0}$ is finely open.  
By \cite[Lemma
5.1.5(ii)]{FOT}, 
\begin{equation}
\label{eq.wsp.ctd4}
R_{1/n_0}g_{n_0}(x)\ge E_x\int_0^\zeta \phi^{1/n_0,g_{n_0}}_{A^l}(X_t)\,dA^l_t,\quad x\in E,\quad l\ge 0.
\end{equation}
So, as in the case of (\ref{eq4.01}), we get
\[
n_0 c_1\ge R_{1/n_0}g_{n_0}(x)\ge E_x\int_0^\zeta \phi^{1/n_0,g_{n_0}}_A(X_t)\,dA^l_t,\quad x\in E,\, l\ge 1.
\]
Hence
\[
\frac{n_0 c_1}{c}\ge E_{x_0}\int_0^\zeta \mathbf{1}_{V_{x_0}}(X_t)\,dA^l_t=
\int_{V_{x_0}} G(x_0,y)\,\nu^l(dy),\quad l\ge 1.
\]
From this we conclude that $x_0\in E_\nu$. Thus,
$\{\phi^\alpha_A>0\}\subset E_\nu$,  which combined with the  reverse inequality proved above  implies (v).  (vi) is a direct consequence of (v). 
\end{proof}

\begin{corollary}
Let $\nu$ be a positive smooth measure on $E$ and  $B$  a PCAF
of $\mathbb X$ with exceptional set $N\subset N_\nu$ such that $B$
is in the Revuz correspondence with $\nu$ and $\phi^1_B$ (cf. \eqref{eq.dofcaf12})
is finely continuous. Then $N=N_\nu$, and for every $x\in E$,
$P_x(A_t=B_t,\, t> 0)=1$, where $A$ is the PCAF constructed in
Proposition \ref{prop4.1}.
\end{corollary}
\begin{proof}
Since $A$ and $B$ satisfy (\ref{eq4.1}) for every $x\in E_\nu$, applying \cite[Proposition IV.2.12]{BG} yields $P_x(A_t=B_t,t\ge0)=1$, $x\in E_\nu$. This implies that $\phi^1_A=\phi^1_B$ on $E_\nu$. Since $\phi^1_A$ and $\phi^1_B$ are finely continuous,  in fact $\phi^1_A=\phi^1_B$. Thus $P_x(B_t=\infty, t>0)=1,\, x\in N_\nu$, which implies that $N=N_\nu$ and $P_x(A_t=B_t, t>0)=1$, $x\in E$.
\end{proof}

\begin{corollary}
Define $E^\alpha_\nu$ by \mbox{\rm(\ref{eq3.1})} but with $G$ replaced by $G_{\alpha}$
%\[
%E^\alpha_\nu=\{x\in E: \exists\, V_x\mbox{-finely open neighborhood of  $x$ such that } %\int_{V_x} G_\alpha(x,y)\,\nu(dy)<\infty\}.
%\]
Then for every $\alpha>0$, $E^\alpha_\nu=E_\nu$.
\end{corollary}

From now on, for a given smooth measure $\nu$, we denote by $A^\nu$ the PCAF of $\mathbb X$ with exceptional set $N_\nu$ constructed in Proposition \ref{prop4.1}.

We close this section with a remark concerning the case when $\mathbb X$ is a Brownian motion.  In that case Baxter, Dal Maso and Mosco \cite{BDM}
and Sturm \cite{Sturm} have shown a duality between positive Borel measures absolutely continuous with respect to Newtonian capacity (i.e. satisfying only condition (a) of the definition of a smooth measure) and  so called positive additive functionals (PAFs) of $\mathbb X$, i.e.  $\mathbb F$-adapted right continuous processes $A:\Omega\times [0,\infty)\rightarrow [0,\infty]$ such that
\begin{enumerate}[(a)]
\item $\forall _{s,t\ge 0}\quad A_{t+s}=A_t+A_s\circ\theta_t$ a.s.,

\item $\forall_{s\ge 0}\quad t\mapsto A_{s-t}\circ \theta_t$ is a.s. right-continuous on $[0,s]$.
\end{enumerate}
Observe that for given smooth measure $\nu$, we have  that  $\tilde A^\nu_t:= A^\nu_{t+}$, $t\ge 0$ is a PAF of $\mathbb X$  in Revuz duality with $\nu$
considered in \cite{BDM} and \cite{Sturm}.
So, in other words, $\tilde A^\nu$ is the unique PAF of $\mathbb X$ associated with $\nu$.

\section{Duality solutions to Schr\"odinger equations }

Let $\nu$ be a positive smooth measure on $E$ and
$(\EE^\nu,D(\EE^\nu))$ be a quasi-regular Dirichlet form on
$L^2(E;m)$ being the perturbation of a regular Dirichlet form
$(\EE,D(\EE))$ by a measure $\nu$. We denote by $-A+\nu$ the
self-adjoint operator on $L^2(E;m)$  generated by
$(\EE^\nu,D(\EE^\nu))$, and by $(T^\nu_t)_{t\ge0}$ the Markov
semigroup of contractions on $L^2(E;m)$ generated by $-A+\nu$. The
resolvent determined by $(T^\nu_t)_{t\ge0}$  shall be denoted by
$(J^\nu_\alpha)_{\alpha>0}$.

In \cite{MO} the Schr\"odinger equation (\ref{eq1.1}) with $A$
defined by (\ref{eq1.7}), and $\mu$ having compact support is considered.  Let  $a$ be a symmetric  matrix-valued bounded Borel
measurable function on a bounded domain $D$ such that $\lambda
I\le a$ for some $\lambda>0$. In \cite[Definition 5.4]{MO} the following
definition of a solution is adopted:  $u\in L^1(D;m)$ is  a
 duality solution to (\ref{eq1.1}) in the sense of \cite{MO} if
\begin{equation}
\label{eq4.s1}
\langle u,\eta\rangle=\langle \mu,\hat\zeta_\eta\rangle,
\quad \eta\in L^\infty(D;m),
\end{equation}
where $\zeta_\eta$ is the unique minimizer of the energy functional
\[
E(u)=\frac{1}{2}\int_D|\sigma\nabla u|^2\,dm
+\frac12\int_D\tilde u^2\,d\nu-\int_D  u\eta\,dm,
\quad u\in H^1_0(D)\cap L^2(D;\nu),
\]
$\sigma\cdot \sigma^*=a$, and $\hat\zeta_{\eta}$ is defined as
\[
\hat\zeta_\eta(x)=\lim_{r\rightarrow 0^+}
\frac{1}{|B(x,r)|}\int_{B(x,r)}\zeta_\eta(y)\,dy,\quad x\in D.
\]
The limit above is well defined in each point $x\in D$ because
$\zeta_\eta$ is a difference  of bounded superharmonic functions, so each
point in $D$ is a Lebesgue point for $\zeta_\eta$. This definition
is a generalization of the notion of  solution considered in
\cite{DM1,DM} in case $A$ is defined by (\ref{eq1.7}) and the
measure $\mu$ belongs to $H^{-1}(D)$. Under this additional
assumption on $\mu$, $u$ is a duality solution  to (\ref{eq1.1})
if and only if $u\in H^1_0(D)$ and for every $\eta\in H^1_0(D)\cap
L^2(E;\nu)$,
\[
\int_Da \nabla u\nabla \eta+\int_D \tilde u\tilde \eta \,d\nu
=\langle \mu,\eta\rangle_{H^{-1}(D),H^1_0(D)},
\]
see \cite[Remark 5.5]{MO}.

The goal of this section is to extend the notion of  duality
solutions to Schr\"odinger equation (\ref{eq1.1}) with  general
Dirichlet operator $A$. Before giving the rigorous definition,
some  remarks are in order. Observe that in fact, $\zeta_\eta=J^\nu\eta,\, m$-a.e., so the
key in the definition of duality solution is to give a proper
pointwise meaning to the function $J^\nu\eta$.  For general $A$, we
can no longer apply the notion of Lebesgue's points. Instead,  we
use the notion of Green's function.  The problem is that without
additional assumptions on the measure $\nu$, in general there is
no Green function $G^\nu$ for the operator $-A+\nu$ on the whole
$E$. To overcome  this difficulty, we use the fact that there
always exists a Green function $G^{E_\nu,\nu}$ on $E_\nu$ (cf. \eqref{eq3.1}) for the
operator $(-A+\nu)_{|E_\nu}$ being the restriction of $-A+\nu$ to
$E_\nu$, and then we show how to extend $G^{E_\nu,\nu}$ in a
canonical way to the whole $E$.

\subsection{Existence and uniqueness of  duality solutions}
\label{sec4.1}

By \cite[Theorem A.2.11]{FOT}, there exists a Hunt process
\[
\mathbb X^\nu=((P^\nu_x)_{x\in E\cup \{\Delta\}},\,\mathbb F^\nu:=\{\FF^\nu_t,\, t\ge 0\})
\]
associated with the form $(\EE^\nu,D(\EE^\nu))$ in the sense that  for every $f\in \BB(E)\cap L^2(E;m)$,
\begin{equation}
\label{eq4.s23}
T^\nu_tf(x)=E^\nu_xf(X_t)\quad m\mbox{-a.e. }x\in E.
\end{equation}
We set
\[
P^\nu_tf(x)= E^\nu_xf(X_t),\qquad R^\nu f(x)= E^\nu_x\int_0^\zeta f(X_t)\,dt,\quad x\in E,
\]
where $E^{\nu}_x$ stands for the expectation with respect to $P^{\nu}_x$. By the construction of the process $\mathbb X^\nu$,
\begin{equation}
\label{eq4.s21}
R^\nu f(x)=E_x\int_0^\zeta e^{-A^\nu_t}f(X_t)\,dt,\quad x\in E_\nu.
\end{equation}
Since $A^\nu$ is a strict PCAF of $\mathbb X^{E_\nu}$,  by
\cite[Exercise 6.1.1]{FOT} the process $\mathbb X^{E_\nu,\nu}$
satisfies the absolute continuity condition. Therefore there
exists a Green function $G^{E_\nu,\nu}$ on $E_\nu\times E_\nu$
associated with the process $\mathbb X^{E_\nu,\nu}$, i.e. for
every  $\eta\in\BB^+(E)$,
\begin{align}
\label{eq4.s2ab}
E^{E_\nu,\nu}_x\int_0^{\zeta} \eta(X_t)\,dt=\int_{E_\nu} G^{E_\nu,\nu}(x,y)\eta(y)\,m(dy),\quad x\in E_\nu.
\end{align}
Moreover, by \cite[Exercise 6.1.1]{FOT} again,
\begin{equation}
\label{eq4.4}
G^{E_\nu,\nu}(x,y)+\int_{E_\nu}
G^{E_\nu}(x,z)G^{E_\nu,\nu}(z,y)\,\nu(dz)=G^{E_\nu}(x,y),\quad x,y\in E_\nu.
\end{equation}
Observe that by \eqref{eq3.0sd} and \eqref{eq4.s21},
\begin{align}
\label{eq4.s2bc}
E_x\int_0^{\tau_{E_\nu}} e^{-A^\nu_t} \eta(X_t)\,dt=E^{E_\nu,\nu}_x\int_0^{\zeta} \eta(X_t)\,dt,\quad x\in E_\nu.
\end{align}
Since $N_\nu$ is a polar set for $\mathbb X$, $P_x(\tau_{E_\nu}=\zeta)=1,\, x\in E_\nu$.
Thus,
\begin{equation}
\label{eq4.s22}
E_x\int_0^{\tau_{E_\nu}} e^{-A^\nu_t} \eta(X_t)\,dt=E_x\int_0^{\zeta} e^{-A^\nu_t} \eta(X_t)\,dt,\quad x\in E_\nu.
\end{equation}
From this and (\ref{eq4.s2ab}), \eqref{eq4.s2bc}, we conclude that for $f\in\BB^+(E)$,
\begin{equation}
\label{eq4.s3}
R^\nu f(x)=\int_{E_\nu} G^{E_\nu,\nu}(x,y)f(y)\,m(dy)=R^{E_\nu,\nu}f(x),\quad x\in E_\nu.
\end{equation}
Observe also, that by \eqref{eq4.4} and the symmetry of $G^{E_\nu}, G^{E_\nu,\nu}$, we have that for any positive 
smooth measures $\nu,\beta$,
\begin{equation}
\label{eq4.4najsbfj90639}
R^\nu(R\beta\cdot\nu)=R(R^\nu\beta\cdot \nu),\quad\mbox{on}\quad E_\nu.
\end{equation}

A careful look  at the construction of the process $\mathbb X^\nu$
(see the comments before \cite[Theorem 6.1.1]{FOT}) reveals  that
\[
R^\nu f(x)=\infty\cdot f(x),\quad x\in N_\nu.
\]
We set
\[
\check{R}^\nu f(x)=\mathbf{1}_{E_\nu} R^\nu f(x),\quad x\in E.
\]
By (\ref{eq4.s21}) and Proposition \ref{prop4.1},
\begin{equation}
\label{eq4.s222}
\check{R}^\nu f(x)=E_x\int_0^\zeta e^{-A^\nu_t}f(X_t)\,dt,\quad x\in E.
\end{equation}
Define
\begin{equation}
\label{eq4.s4}
\check{G}^\nu(x,y)= \left\{
\begin{array}{l} G^{E_\nu,\nu}(x,y), \quad
(x,y)\in E_\nu\times E_\nu,\smallskip \\
0, \quad\quad\quad\quad\,\,\quad (x,y)\in (N_\nu\times E)\cup (E\times N_\nu).
\end{array}
\right.
\end{equation}
With the above notation, we have
\[
\check R^\nu f(x)=\int_{E} \check{G}^{\nu}(x,y)f(y)\,m(dy),\quad x\in E.
\]
We can now extend $\check R^\nu$ to an arbitrary positive Borel measure $\mu$ on $E$ by putting
\begin{equation}
\label{eq4.s5}
\check R^\nu \mu (x)=\int_{E} \check{G}^{\nu}(x,y)\,\mu(dy),\quad x\in E.
\end{equation}
Of course, if $\nu$ is a strictly smooth measure, then $R^\nu=\check R^\nu$ and $G^\nu=\check G^\nu$ since $N_\nu=\emptyset$. Observe also that
\begin{equation}
\label{eq4.s55}
\check R^\nu\le R^\nu\le R.
\end{equation}
By analogous arguing to that  following \eqref{eq.efmd12}, we get that for any positive $\mu$, $\check R\mu$
is excessive with respect to $\mathbb X^{E_\nu,\nu}$.

\begin{definition}
We say that $u\in\BB(E)$ is a {\em duality solution} to (\ref{eq1.1}) if for every $\eta\in \BB(E)$ such that $R|\eta|$ is bounded we have
\begin{equation}
\label{eq4.s51}
\langle u,\eta\rangle= \langle \mu, \check R^\nu\eta\rangle.
\end{equation}
\end{definition}

\begin{remark}
\label{rem4.sr}
Observe that both integrals in (\ref{eq4.s51}) are well defined. Indeed, we have
\[
\langle |\mu|,|\check R^\nu\eta|\rangle \le \langle |\mu|,R|\eta|\rangle \le \|\mu\|_{TV}\|R|\eta|\|_\infty.
\]
Moreover,
\[
\langle |u|,|\eta|\rangle= \langle u,\mbox{sgn} (u)|\eta|\rangle=\langle \mu,\check R^\nu(\mbox{sgn}(u)\eta)\rangle\le \|\mu\|_{TV}\|R|\eta|\|_\infty.
\]
Note also  that thanks to the assumption that $(\EE,D(\EE))$ is transient there exists a strictly positive Borel function $\eta$ on $E$ such that $R\eta$ is bounded (see \cite[Corollary 1.3.6]{Oshima}).
\end{remark}

\begin{theorem}
\label{th4.s12}
Assume that $\nu$ is a positive smooth measure on $E$ and $\mu$ is a bounded Borel measure on $E$.
Then there exists a unique duality solution $u$ to \mbox{\rm(\ref{eq1.1})}. Moreover, there exists a quasi-continuous function  $\tilde u$
which is an  $m$-version of $u$ and 
\begin{equation}
\label{eq4.s52}
\tilde u(x)=\int_E \check G^\nu(x,y)\,\mu(dy)\quad \mbox{q.e}.
\end{equation}
%is an $m$-version of $u$ which is  finely continuous on $E_\nu$.
\end{theorem}
\begin{proof}
By \cite[Proposition 3.2]{K:CVPDE}, for every bounded Borel measure $\mu$, $R|\mu|$ is finite q.e.
Since $\check R^\nu|\mu|\le R|\mu|$, we see that  $\check R^\nu|\mu|$ is finite q.e.
Let $\tilde u(x)=\check R^\nu \mu(x)$ for $x\in E$ such that $R|\mu|(x)<\infty$ and $\tilde u(x)=0$ otherwise.
$\check R^\nu\mu^+$ and $\check R^\nu\mu^-$ are excessive functions with respect to $\mathbb X^{E_\nu,\nu}$ (see the comment following \eqref{eq4.s55}), so they are finely continuous with respect to $\mathbb X^{E_\nu,\nu}$ (see Section \ref{sec2.2}).
In particular, by  \eqref{eq4.s2bc}  processes $e^{-A^\nu}\check R^\nu\mu^+(X),\, e^{-A^\nu}\check R^\nu\mu^-(X)$ are right-continuous under measure $P_x$
for every $x\in E_\nu$. Thus, by continuity of $A^\nu$ and the definition of $E_\nu$, $\check R^\nu\mu^+(X),\,\check R^\nu\mu^-(X)$ are right-continuous under measure $P_x$ for every $x\in E_\nu$. By \cite[Theorem 4.6.1]{FOT}, $\check R^\nu\mu^+, \check R^\nu\mu^-$ are quasi-continuous. Since $\tilde u=\check R^\nu\mu^+-\check R^\nu\mu^-$ q.e., we get that $\tilde u$ is quasi-continuous too.
Let $\eta\in \BB(E)$ be a  function such that $R|\eta|$ is bounded. Then
\[
\langle \tilde u,\eta\rangle= \langle \check R^\nu\mu^+,\eta\rangle-\langle \check R^\nu\mu^-,\eta\rangle=
\langle \mu^+, \check R^\nu\eta\rangle-\langle \mu^-,\check R^\nu\eta\rangle=\langle \mu,\check R^\nu\eta\rangle.
\]
Therefore $\tilde u$ is a duality solution to (\ref{eq1.1}) and of course (\ref{eq4.s52}) is satisfied.
Now, suppose that $u,w$ are duality solutions to (\ref{eq1.1}). Let $\eta$ be a strictly positive Borel function  on $E$ such that $R\eta$ is bounded. Then, by (\ref{eq4.s51}),
$\langle u-w,\mbox{sgn}(u-w)\eta\rangle=0$.
This implies that $\langle |u-w|,\eta\rangle=0$, hence that  $u=w$ $m$-a.e.
\end{proof}

To show that for $A$ defined by (\ref{eq1.7}) duality solutions
coincide with solutions defined by (\ref{eq4.s1}) we will need the
following proposition which is a generalization of \cite[Theorem 5.2]{MO}.

\begin{proposition}
\label{eq4.prop.s3} For every  $\eta\in\BB^+(E)$ such that $R\eta$
is bounded  there exists a unique positive smooth measure
$\gamma_\eta$ such that
\[
\check R^\nu\eta+R\gamma_\eta=R\eta.
\]
Moreover, $\gamma_\eta=(R^\nu\eta)\cdot \nu$.
\end{proposition}
\begin{proof}
%By Theorem \ref{prop4.1}, $A^\nu$ is a strict PCAF of $\mathbb X^{E_\nu}$. 
Multiplying both sides of \eqref{eq4.4} by $\eta$  and
integrating with respect to $y$ over $E_\nu$ yields
\[
R^{E_\nu,\nu}\eta(x)=R^{E_\nu}\eta(x)-R^{E_\nu}((R^{E_\nu,\nu}
\eta)\cdot\nu)(x),\quad x\in E_\nu.
\]
By (\ref{eq4.s5}), $R^{E_\nu,\nu}\eta(x)=\check R^\nu\eta(x),\, x\in E_\nu$. By (\ref{eq3.s1}), $R^{E_\nu}\eta(x)=R\eta(x),\, x\in E_\nu$. By (\ref{eq3.s2}) and (\ref{eq4.s3}),
\begin{align*}
R^{E_\nu}((R^{E_\nu,\nu}\eta)\cdot\nu)(x)
&=\int_{E_\nu} R^{E_\nu,\nu}\eta(y)G^{E_\nu}(x,y)\,\nu(dy)\\
&=\int_{E_\nu}R^\nu\eta(y)G(x,y)\,\nu(dy)\\
&=R(\mathbf{1}_{E_\nu}(R^\nu\eta)\cdot\nu)(x)
=R((R^\nu\eta)\cdot\nu)(x),\quad x\in E_\nu.
\end{align*}
In the last equation we have used the fact that $\nu$ is smooth,
which implies that $\nu(N_\nu)=0$. By the above equations,
\[
\check R^\nu\eta(x)+R\gamma_\eta(x)=R\eta(x),\quad x\in E_\nu,
\]
with $\gamma_\eta= R^\nu\eta\cdot\nu$. Since $\check R^\nu\eta, R\gamma_\eta, R\eta$
are finely continuous (see Proposition \ref{prop4.1}(ii)) and $N_\nu$ is polar, we get the desired result.
\end{proof}

\begin{corollary}
Let  $a$ be a symmetric matrix-valued bounded Borel measurable function on a bounded domain $D$ such that $\lambda I\le a$ for some $\lambda>0$, and let $A$ be defined by \mbox{\rm(\ref{eq1.7})}.
Then $u$ is a duality solution to \mbox{\rm(\ref{eq1.1})} if and only if $u$ satisfies \mbox{\rm(\ref{eq4.s1})}.
\end{corollary}
\begin{proof}
Follows from Proposition \ref{eq4.prop.s3} and \cite[Theorem 5.2]{MO}.
\end{proof}

\subsection{Duality solutions vs. variational solutions
and stability results}
\label{sec4.2}

We start with some stability results for duality solutions.

\begin{proposition}
\label{prop4.s1} Let $\nu$ be a positive smooth measure on $E$ and
$\{\nu_n\}$ be a sequence of  positive strictly smooth measures on
$E$ such that $\nu_n\nearrow \nu$. Then for every $\eta\in
\BB^+(E)$ such that $R\eta$ is bounded,
\[
R^{\nu_n}\eta(x)\rightarrow \check R^\nu\eta(x),\quad x\in E.
\]
\end{proposition}
\begin{proof}
By (\ref{eq4.s21}), for any $\eta$ as in the proposition, we have
\[
R^{\nu_n}\eta(x)=E_x\int_0^\zeta e^{-A^{\nu_n}_t}\eta(X_t)\,dt,\quad x\in E.
\]
By our assumptions and Proposition \ref{prop4.1},  $A^{\nu_n}_t\nearrow A^\nu_t,\, t\ge 0$, $P_x$-a.s. for every $x\in E$. Hence
\[
E_x\int_0^\zeta e^{-A^{\nu_n}_t}\eta(X_t)\,dt\rightarrow E_x\int_0^\zeta e^{-A^{\nu}_t}\eta(X_t)\,dt,\quad x\in E.
\]
Therefore, by \eqref{eq4.s222}, we get the result.
%By Theorem \ref{prop4.1}, $P_x(A^\nu_t=\infty,\, t>0)=1$, $x\in N_\nu$, so $E_x\int_0^\zeta e^{-A^{\nu}_t}\eta(X_t)\,dt=0,\, x\in N_\nu$. If $x\in E_\nu$, then $\zeta=\tau_{E_\nu}$ since $N_\nu$ is polar.
%Hence, by (\ref{eq4.s2ab}) and (\ref{eq4.s22}),
%\[
%E_x\int_0^\zeta e^{-A^\nu_t}\eta(X_t)\,dt=\int_E G^{E_\nu,\nu}(x,y)\eta(y)\,m(dy),\quad x\in E_\nu,
%\]
%which completes the proof.
\end{proof}

\begin{proposition}
\label{prop4.s1s}
Let $\nu$ be a positive smooth measure on $E$ and $\{\nu_n\}$ be a sequence of positive strictly smooth measures on $E$ such that $\nu_n\nearrow \nu$.
Let $u$ be a duality solution to \mbox{\rm(\ref{eq1.1})} and $u_n$ be a duality solution to \mbox{\rm(\ref{eq1.1})}  with $\nu$ replaced by $\nu_n$. Then, for every  $\rho\in \BB^+(E)$ such that $R\rho$ is bounded, $u_n\rightarrow u$ in $L^1(E;\rho\cdot m)$.
\end{proposition}
\begin{proof}
Set $u_n^{\oplus}= R^{\nu_n}\mu^+,\, u_n^{\ominus}= R^{\nu_n}\mu^-$.
Then $u_n^{\oplus}, u_n^{\ominus}$ are  excessive functions with respect to $\mathbb X^{\nu_n}$. Therefore, by (\ref{eq4.s21}),
\[
\alpha R^\nu_\alpha u_n^{\oplus}(x)\le \alpha R_\alpha^{\nu_n} u_n^{\oplus}(x)\le u_n^{\oplus}(x),\quad x\in E_\nu.
\]
This implies that $\{u_n^{\oplus}\}$ is a sequence of  excessive functions  with respect to $\mathbb X^{\nu}$.
Analogously, we get that $\{u_n^{\ominus}\}$ is also  a sequence of  excessive functions  with respect to $\mathbb X^{\nu}$.
Since there exists the Green function for $\mathbb X^{E_\nu,\nu}$, by \cite[Lemma 94, page 306]{DellacherieMeyer}, there exists a subsequence (still denoted by $(n)$) such that $\{u_n^{\oplus}\}, \{u_n^{\ominus}\}$ are convergent $m$-a.e.  Observe that
\[
|u_n|\le R|\mu|,\quad m\mbox{-a.e.}
\]
Hence, by the Lebesgue dominated convergence theorem, there exists $u\in L^1(E;\rho\cdot m)$ such that $\{u_n\}$  converges to $u$ in $L^1(E;\rho\cdot m)$. As a consequence,   for every $\eta\in \BB(E)$ such that $R|\eta|$ is bounded,
\[
\langle u_n,\eta\rangle\rightarrow \langle u,\eta\rangle.
\]
By Proposition \ref{prop4.s1},
\[
\langle \mu, R^{\nu_n}\eta\rangle \rightarrow \langle \mu,\check R^\nu\eta\rangle.
\]
From these two convergences, we conclude that $u$ is a duality solution to (\ref{eq1.1}).
Applying a  uniqueness argument  (see Theorem \ref{th4.s12}) shows that in fact the whole sequence $\{u_n\}$ converges  to $u$ in $L^1(E;\rho\cdot m)$.
\end{proof}

%\begin{corollary}
%Let $\nu$ be a positive smooth measure on $E$. Let $\{\nu_n\}$ be a sequence of positive strict smooth measures on $E$ such that $\nu_n\nearrow \nu$. Then the assertions of Proposition \ref{prop4.s1} and  Proposition \ref{prop4.s1s} hold.
%\end{corollary}

Since in the paper we assume that $(\EE,D(\EE))$ is transient, $(D_e(\EE),\EE)$
is a Hilbert space.
In what follows we denote by  $D'_e(\EE)$ the dual space of $D_e(\EE)$.
Let $\mu$ be a bounded Borel measure on $E$. We write $\mu\in D'_e(\EE)$ if
\[
\langle \mu,\eta\rangle\le c\|\eta\|_{D_e(\EE)},\quad \eta\in C_b(E)\cap D_e(\EE).
\]
By \cite[Proposition 3.1]{KR:BPAN}, each bounded Borel measure $\mu\in D'_e(\EE)$ is a smooth measure. Moreover, for every bounded $\eta\in D_e(\EE)$,
\[
\langle \mu,\tilde \eta\rangle\le c\|\eta\|_{D_e(\EE)}.
\]
By the Hahn-Banach theorem,  $\mu$ can be uniquely extended to $D_e(\EE)$
as a continuous linear functional. We denote this extension again by $\mu$. With this notation,  for every bounded $\eta\in D_e(\EE)$,
\begin{equation}
\label{eq4.s5352}
\langle\mu,\eta\rangle_{D'_e(\EE),D_e(\EE)}=\int_E\tilde\eta\,d\mu.
\end{equation}

\begin{definition}
Let $\nu$ be a positive smooth measure and $\mu$ be a bounded Borel measure in $D'_e(\EE)$.
We say that $u\in D_e(\EE)\cap L^2(E;\nu)$ is a {\em variational solution} to (\ref{eq1.1}) if
\[
\EE(u,\eta)+\langle \nu,\tilde u \tilde\eta\rangle=\langle \mu,\tilde \eta\rangle_{D'_e(\EE),D_e(\EE)},\quad \eta\in D_e(\EE)\cap L^2(E;\nu).
\]
\end{definition}

\begin{remark}
Since $D_e(\EE^\nu)\subset D_e(\EE)$, the existence and uniqueness of a variational solution to (\ref{eq1.1})
follows easily from the  Lax-Milgram theorem.
\end{remark}

\begin{theorem}
Let $\nu$ be a positive smooth measure on $E$ and $\mu$ be a bounded
Borel measure on $E$ such that $\mu\in D'_e(\EE)$. Then $u$ is  a
variational solution to \mbox{\rm(\ref{eq1.1})} if and only if it
is a duality solution to \mbox{\rm(\ref{eq1.1})}.
\end{theorem}
\begin{proof}
By the existence and uniqueness results for variational and duality solutions to (\ref{eq1.1}) it is enough
to prove that if $u$ is a variational solution to (\ref{eq1.1}), then $u$ is a duality solution to (\ref{eq1.1}).
Suppose that  $u\in D_e(\EE)\cap L^2(E;\nu)$ is a variational solution to (\ref{eq1.1}).  Then, by the definition,
\begin{equation}
\label{eq4.s53}
\EE^\nu(u,v)=\langle \mu, \tilde v\rangle_{D'_e(\EE),D_e(\EE)},\quad v\in D_e(\EE)\cap L^2(E;\nu).
\end{equation}
Let $\eta\in\BB^+(E)$ be such that  $R\eta$ is bounded. By
\cite[Lemma 2.1]{KR:BPAN}, there exists a generalized  nest
$\{F_n\}$ such that $\eta_n:=\mathbf{1}_{F_n}\eta\in
D_e'(\EE^\nu)$. We have $v_n:= \check R^\nu\eta_n\in
D_e(\EE^\nu)=D_e(\EE)\cap L^2(E;\nu)$. Taking $v_n$ as a test
function in (\ref{eq4.s53}) and using (\ref{eq4.s5352}) we get
\[
\langle u,\eta_n\rangle=\langle \mu,\widetilde{\check R^\nu\eta_n}\rangle.
\]
By \eqref{eq4.s222} and Proposition \ref{prop4.1}(ii), $\check R^\nu\eta_n$ is
finely continuous, so by \cite[Theorem 4.2.2]{FOT} it is
quasi-continuous. Thus $\widetilde{\check R^\nu\eta_n}=\check
R^\nu\eta_n$ q.e. Since $\mu$ is smooth, we conclude that
\[
\langle u,\eta_n\rangle=\langle \mu,\check R^\nu\eta_n\rangle.
\]
Letting $n\rightarrow\infty$ in the above equation  (see Remark
\ref{rem4.sr})  yields the desired result.
\end{proof}

Now we are going to show that in some sense the notion of  duality
solution to Schr\"odinger equation (\ref{eq1.1}) is natural. To
be precise, we will show that  each duality solution to
(\ref{eq1.1}) (with measure $\mu$ on the right-hand side) is a
limit of variational solutions to (\ref{eq1.1}) with  suitable chosen   $\mu_n$
$\in D'_e(\EE)$  approximating the  measure $\mu$ in the narrow
topology, i.e.
\[
\int_E\eta\,d\mu_n\rightarrow \int_E\eta\,d\mu,\quad \eta\in C_b(E).
\]

Let $\mu$ be a Borel measure on $E$. In the sequel, $\|\mu\|_{TV}$
stands for its total variation norm.

\begin{proposition}
\label{prop.4.stw4} Let $\nu$ be a positive smooth measure on $E$,
$\mu$ be a  bounded Borel measure on $E$, and let $u$ be a duality
solution to \mbox{\rm(\ref{eq1.1})} and $u_n$ be a duality
solution to \mbox{\rm(\ref{eq1.1})} with $\mu$ replaced by
$\mu_n:=nR_n\mu$. Then $u_n\rightarrow u$ in $L^1(E;\rho\cdot m)$
for every strictly positive Borel function $\rho$ on $E$ such that
$R\rho$ is bounded.
\end{proposition}
\begin{proof}
Let $\eta\in\BB^+(E)$ be such that $R\eta$ is bounded. Since
$\check R^\nu\eta\le R\eta$, $\check R^\nu\eta$ is bounded. This,
when combined with the fact that $\check R^\nu\eta$ is finely
continuous, implies that $nR_n(\check R^\nu\eta)\rightarrow \check
R^\nu\eta$.  We have  $n R_n(\check R^\nu\eta)\le
\|R\eta\|_\infty$. Hence
\[
\langle \mu, nR_n(\check R^\nu\eta)\rangle
\rightarrow \langle \mu,\check R^\nu\eta\rangle.
\]
Therefore
\[
\langle \mu_n, \check R^\nu\eta\rangle =\langle \mu,nR_n(\check
R^\nu\eta)\rangle\rightarrow \langle \mu, \check R^\nu\eta\rangle.
\]
By the definition of a duality solution,
\begin{equation}
\label{eq.wig1}
\langle u_n,\eta\rangle=\langle \mu_n, \check R^\nu\eta\rangle.
\end{equation}
Since duality solutions are unique, to complete the proof it is
enough to show that, up to a subsequence, $\{u_n\}$ is convergent
in $L^1(E;\rho\cdot m)$. To show this, we first observe that $u_n$
is an excessive function with respect to $\mathbb X^{E_\nu,\nu}$
since $u_n= \check R^\nu(\mu_n)$. Since $\mathbb X^{E_\nu,\nu}$
satisfies the absolute continuity condition, by \cite[Lemma 94, p.
306]{DellacherieMeyer}, there exists a subsequence (still denoted
by $(n)$) such that $\{u_n\}$ is convergent $m$-a.e. Moreover,
\[
|u_n|=|\check R^\nu\mu_n|=|\check R^\nu(nR_n\mu)|
\le  R(nR_n|\mu|)=nR_n(R|\mu|)\le R|\mu|.
\]
In the last inequality, we used the fact that $R|\mu|$ is  an
excessive function. Observe that
\[
\langle R|\mu|,\rho\rangle=\langle |\mu|, R\rho\rangle \le \|\mu\|_{TV}\|R\rho\|_\infty.
\]
Therefore, applying the Lebesgue dominated convergence  to \eqref{eq.wig1} yields the desired result.
\end{proof}

\begin{corollary}
\label{cor.cor4.1} Let $\rho$ be a strictly positive Borel
function on $E$ such that $R\rho$ is bounded. There exists a
sequence $\{\mu_n\}\subset D_e'(\EE)\cap L^2(E;m)$ such that
$\mu_n\rightarrow \mu$ in the narrow topology and $u_n\rightarrow
u$ in $L^1(E;\rho\cdot m)$,  where $u_n$ is the variational
solution to \mbox{\rm(\ref{eq1.1})} with $\mu$ replaced by
$\mu_n$ and $u$ is the duality solution to
\mbox{\rm(\ref{eq1.1})}.
\end{corollary}
\begin{proof}
Let $u$ be the duality solution to (\ref{eq1.1}). Set $\mu_n=
nR_n\mu$ and let $\{F_{n,k}\}_{k\ge1}$ be a generalized nest such
that $\mu_{n,k}:=\mathbf{1}_{F_{n,k}}\cdot \mu_n\in D_e'(\EE),\,
k\ge 1$. It is clear that $\|\mu_{n,k}-\mu_n\|_{TV}\rightarrow 0$
as $k\rightarrow \infty$. Let $u_{n,k}$ be a duality solution to
(\ref{eq1.1}) with $\mu$ replaced by $\mu_{n,k}$ and  $u_n$ be
a duality solution to (\ref{eq1.1}) with $\mu$ replaced by
$\mu_n$. Observe that
\[
|u_{n,k}-u_n|\le R|\mu_{n,k}-\mu_n|,\quad m\mbox{-a.e}.
\]
Let $k_n\in\mathbb N$ be such that $\|\mu_{n,k_n}-\mu_n\|_{TV}\le1/n$.
Then
\[
\|u_{n,k_n}-u_n\|_{L^1(E;\rho\cdot m)}\le \|R\rho\|_\infty \|\mu_{n,k_n}-\mu_n\|_{TV}\le \|R\rho\|_\infty/n.
\]
By Proposition \ref{prop.4.stw4}, $u_{n}\rightarrow u$ in $L^1(E;\rho\cdot m)$   as $n\rightarrow \infty$. Consequently, $\|u_{n,k_n}-u\|_{L^1(E;\rho\cdot m)}\rightarrow 0$. Now, we shall  prove  that $\mu_{n,k_n}\rightarrow \mu$ in the narrow topology.
Let $\eta$ be a bounded continuous function on $E$. Then
\begin{align*}
|\langle \mu_{n,k_n}-\mu,\eta\rangle|&\le |\langle \mu_{n,k_n}-\mu_n,\eta\rangle|+|\langle \mu_{n}-\mu,\eta\rangle|\\
&\le \|\eta\|_\infty\|\mu_{n,k_n}-\mu_n\|_{TV}+|\langle \mu_{n}-\mu,\eta\rangle|\\
&\le\|\eta\|_\infty\|R\rho\|_\infty/n+|\langle \mu_{n}-\mu,\eta\rangle|.
\end{align*}
Since $|\langle \mu_{n}-\mu,\eta\rangle|=|\langle\mu,nR_n\eta-\eta\rangle|$
converges to zero as $n\rightarrow\infty$, this shows that the sequence $\{\mu_{n,k_n}\}$ has the desired properties.
\end{proof}

\subsection{Regularity results for duality solutions}
\label{sec4.3}
For $k\ge 0$, we denote
\[
T_k(u)=\min\{k, \max\{u,-k\}\}.
\]

\begin{lemma}
\label{lm.srefe033678}
For any excessive function $\rho$ and positive smooth measure $\nu$,
\[
\check R^\nu(\rho\cdot \nu)\le \rho,\quad \mbox{on}\quad E.
\]
\end{lemma}
\begin{proof}
Let $\rho$ and $\nu$ be as in the assertion of the lemma. 
By \cite[Proposition II.2.6]{BG} (see also \cite[Exercise II.2.19]{BG}), there exists a sequence $\{\eta_n\}\subset\BB^+_b(E)$ such that $\rho_n:= R\eta_n\nearrow \rho$.
By \eqref{eq4.4najsbfj90639} and \cite[Lemma 5.1.5(ii)]{FOT},
\[
R^\nu(\rho_n\cdot\nu)=R(R^\nu\eta_n\cdot\nu)\le R\eta_n=\rho_n,\quad\mbox{on}\quad E_\nu.
\]
Letting $n\rightarrow \infty$ and using \eqref{eq4.s4}, \eqref{eq4.s55}, we get the result.
\end{proof}

\begin{theorem}
\label{th4.3.dsu}
Let $u$ be a duality solution to \mbox{\rm(\ref{eq1.1})}. Then
\begin{enumerate}[\rm(i)]
\item $u$ has a  quasi-continuoius $m$-version $\tilde u$ such that
$\tilde u(x)=\int_E \check G^\nu(x,y)\,\mu(dy)$ q.e.

\item $\tilde u\in L^1(E;\nu)$ and
$\int_E|\tilde u|\,d\nu\le \|\mu\|_{TV}$.

\item $T_k(u)\in D_e(\EE)$, $k\ge 0$, and
$\EE(T_k(u),T_k(u))\le k\|\mu\|_{TV}$, $k\ge 0$.

%\item $\check u(x)=0,\, x\in N_\nu$.

\item  $|\check u|\le R|\mu|$.

\item If $R(\BB_b(E))\subset \BB_b(E)$, then $u\in L^1(E;m)$.
\end{enumerate}
\end{theorem}
\begin{proof}
Assertion (i) is a consequence of  Theorem \ref{th4.s12}.
Let $\rho$ be an excessive function. By Lemma \ref{lm.srefe033678},
\[
\langle \rho\cdot\nu, |\tilde u| \rangle \le \langle \rho\cdot \nu, \check R^\nu|\mu|\rangle=\langle \check R^\nu(\rho\cdot\nu),|\mu|\rangle\le \langle \rho,|\mu|\rangle,
\]
so we get  (ii) by taking $\rho\equiv 1$. Let $\nu_n$ be a sequence of bounded strictly smooth measures such that $\nu_n\nearrow \nu$. Let $u_n$ be a duality solution to (\ref{eq1.1}) with $\nu$ replaced by $\nu_n$. We have
\[
\tilde u_n(x)=\int_E G^{\nu_n}(x,y)\,\mu(dy),\quad \mbox{q.e}.
\]
Since $\nu_n$ is bounded, $(\EE^{\nu_n}, D(\EE^{\nu_n}))$ is a regular symmetric Dirichlet form. Hence, by \cite[Proposition 5.9]{KR:JFA}, $T_k(u_n)\in D_e(\EE^{\nu_n})$ and
\[
\EE(T_k(u_n),T_k(u_n))\le \EE^{\nu_n}(T_k(u_n),T_k(u_n))\le k\|\mu\|_{TV},\quad k\ge 0.
\]
This when combined with Proposition \ref{prop4.s1s} gives (iii). Assertion (iv) follows from (i) and (\ref{eq4.s55}).
By (iv), we have $|u|\le R|\mu|$ $m$-a.e. Hence
\[
\|u\|_{L^1(E;m)}=\langle |u|,1\rangle\le \langle R|\mu|,1\rangle=\langle |\mu|,R1\rangle\le \|R1\|_\infty\|\mu\|_{TV}.
\]
From this and the inclusion $R(\BB_b(E))\subset \BB_b(E)$ we get (v).
\end{proof}

\section{Strong duality solutions  to Schr\"odinger equations}
\label{sec5}

In this section, we  compare the notion of  duality  solutions to
(\ref{eq1.1}) with the notion of strong duality solutions  to
(\ref{eq1.1}), i.e. solutions to (\ref{eq1.3}). We next provide a
necessary and sufficient condition for the existence of a strong
duality solution to (\ref{eq1.1}). We also give some remarks
concerning the concept of  renormalized solutions.

It is well known (see \cite{FST}) that  each bounded Borel
measure $\mu$ admits a unique decomposition
\[
\mu=\mu_d+\mu_c
\]
into an absolutely continuous with respect to $\mbox{Cap}$ part
$\mu_d$ (called the diffuse part of $\mu$) and an orthogonal to
$\mbox{Cap}$ part $\mu_c$ (called  the concentrated part of
$\mu$).

\begin{definition}
\label{def.sds3054}
Let $\nu$ be a positive smooth measure on $E$ and $\mu$ be a
bounded measure on $E$. We say that a Borel measurable
quasi-continuous function $u$ on $E$ is a {\em strong duality solution} to (\ref{eq1.1})
%\begin{equation}
%\label{eq4.2}
%-Au+u\cdot\nu=\mu,
%\end{equation}
if $u\in L^1(E;\nu)$ and for $m$-a.e. $x\in E$,
\begin{equation}
\label{eq4.2s4} u(x)+\int_Eu(y)G(x,y)\,\nu(dy)=\int_E
G(x,y)\,\mu(dy).
\end{equation}
\end{definition}

\begin{remark}
By  \cite[Proposition 3.2]{K:CVPDE},  both integrals in  (\ref{eq4.2s4}) are well defined for q.e. $x\in
E$. Since in Definition \ref{def.sds3054}, we required from  $u$  to be  quasi-continuous, we have that in fact  (\ref{eq4.2s4}) holds q.e., see Section \ref{sec2.1}.
\end{remark}

\begin{remark}
\label{rem.5.321}
Let $u$ be a strong duality solution to (\ref{eq1.1}). Integrating both sides  of (\ref{eq4.2s4}) with respect to a  smooth measure $\beta$ such that $R|\beta|$ is bounded yields
\begin{equation}
\label{eq5.3}
\langle u,\beta\rangle+\langle u\cdot\nu, R\beta \rangle=\langle \mu, R\beta\rangle.
\end{equation}
Clearly, the above formula gives an equivalent definition of a strong duality solution to (\ref{eq1.1}). In fact this is true if 
(\ref{eq5.3}) is satisfied merely for any  positive Borel function $\beta$  on $E$ such that $R\beta$ is bounded (see \cite{KR:JFA}).
\end{remark}

For a measure $\mu$, we denote by $\mu_{\lfloor E_{\nu}}$ its restriction to the set $E_{\nu}$, where $E_{\nu}$ is defined by (\ref{eq3.1}).

\begin{theorem}
\label{th5.ct.1}
Let $\nu$ be a positive smooth measure on $E$ and $\mu$ be a bounded Borel measure on $E$.
\begin{enumerate}[\rm(i)]
\item If $u$ is a duality solution to \mbox{\rm(\ref{eq1.1})},
then its quasi-continuous $m$-version $\tilde u$ is a strong duality
solution to \mbox{\rm(\ref{eq1.1})} with $\mu$ replaced
by $\mu_{\lfloor E_\nu}$.
\item If $u$ is a strong duality solution to
\mbox{\rm(\ref{eq1.1})}, then $u$ is a duality solution to
\mbox{\rm(\ref{eq1.1})}.
\end{enumerate}
%\[
%u(x)=\int_E \check G^\nu(x,y)\,\mu(dy),\quad q.e.
%\]
\end{theorem}
\begin{proof}
Let $u$ be a duality solution to (\ref{eq1.1}). By  Proposition
\ref{prop4.1}, there exists a  PCAF $A^\nu$ of $\mathbb X$  in the
Revuz duality with $\nu$ with the exceptional set $N_\nu$.  Since
$N_\nu$ is polar for $\mathbb X$, the process $\mathbb X^{E_\nu}$
satisfies the absolute continuity condition and its Green function
$G^{E_\nu}$ satisfies
\begin{equation}
\label{eq4.3}
G^{E_\nu}(x,y)=G(x,y),\quad x,y\in E_\nu
\end{equation}
(see (\ref{eq3.s2}) and the comment following it). Let $\mathbb
X^{E_\nu,\nu}$ be a Hunt process perturbed by the strict PCAF
$A^\nu$ of $\mathbb X^{E_\nu}$. By Theorem \ref{th4.3.dsu}(i)-(ii), $\tilde u\in L^1(E;\nu)$, and 
\[
\tilde u(x)=\int_E \check G^\nu(x,y)\,\mu(dy),\quad \mbox{q.e}.
\]
By the definition of $\check G^\nu(x,y)$,
\begin{equation}
\label{eq4.psf}
\tilde u(x)=\int_{E_\nu} G^{E_\nu,\nu}(x,y)\,\mu(dy),\quad \mbox{q.e.}
\end{equation}
Integrating both sides of (\ref{eq4.4}) with  respect to $\mu(dy)$
over $E_\nu$ yields
\begin{align}
\label{eq4.5} \nonumber \int_{E_\nu}G^{E_\nu,\nu}(x,y)\,\mu(dy)
&+\int_{E_\nu}\Big(
G^{E_\nu}(x,z) \int_{E_\nu} G^{E_\nu,\nu}(z,y)\, \mu(dy)\Big)
\,\nu(dz)\\&\quad =\int_{E_\nu} G^{E_\nu}(x,y)\,\mu(dy)
\end{align}
for $x\in E_{\nu}$.
%Since $|\mu_c|(N_\nu)=0$ and $N_\nu$ is polar, we get $|\mu|(N_\nu)=0$.
From (\ref{eq4.3}), (\ref{eq4.psf}), (\ref{eq4.5}) and smoothness of $\nu$ we conclude that
\[
\tilde u(x)+\int_E G(x,y) \tilde u(y)\,\nu(dy)=\int_E G(x,y)\,\mu_{\lfloor E_\nu}(dy),\quad \mbox{q.e.},
\]
which implies that $\tilde u$ is a strong duality solution to (\ref{eq1.1}) with $\mu$ replaced by $\mu_{\lfloor E_\nu}$.
%By \cite[Theorem 4.2.2]{FOT}, $\check u$ is quasi-continuous, so $\check u=\tilde u$ q.e. This implies that $\tilde u$ is a strong duality solution to (\ref{eq4.2}).

Now suppose that  $u$ is a strong duality solution to
(\ref{eq1.1}).  Then, by Remark \ref{rem.5.321}, for every
smooth measure $\beta$ such that $R|\beta|$ is bounded, we have
\begin{equation}
\label{eq4.6}
\langle u,\beta \rangle+\langle u\cdot\nu,R\beta\rangle
=\langle \mu,R\beta\rangle.
\end{equation}
By Proposition \ref{eq4.prop.s3},
$R\eta=R(\check  R^{\nu}\eta\cdot\nu)+\check R^{\nu}\eta$,
so for every $\eta\in\BB^+(E)$ such that $R\eta$ is bounded,
\[
R(\eta-\check R^{\nu}\eta\cdot\nu)=\check R^{\nu}\eta.
\]
From the above equation and (\ref{eq4.6}) with $\beta=\eta-\check
R^{\nu}\eta\cdot\nu$ we get $\langle u,\eta\rangle=\langle
\mu,\check R^\nu\eta\rangle$,  which shows that $u$ is a duality
solution to (\ref{eq1.1}).
\end{proof}

\begin{corollary}
There exists at most one strong duality solution to
\mbox{\rm(\ref{eq1.1})}.
\end{corollary}
\begin{proof}
Follows from Theorem \ref{th5.ct.1}(ii) and Theorem \ref{th4.s12}.
\end{proof}

\begin{theorem}
\label{th.main.ex.th23}
Let $\nu$ be a positive smooth measure on $E$ and $\mu$ be a
bounded Borel measure on $E$. Then there exists a strong duality
solution to \mbox{\rm(\ref{eq1.1})} if and only if
$|\mu_c|(N_\nu)=0$.
\end{theorem}
\begin{proof}
Assume that $|\mu_c|(N_\nu)=0$. Then, since $N_\nu$ is polar,
$\mu=\mu_{\lfloor E_\nu}$. Therefore, by Theorem \ref{th4.s12} and
Theorem \ref{th5.ct.1}(i), there exists a solution to
(\ref{eq1.1}).

Now assume that that there exists a strong duality  solution $u$
to (\ref{eq1.1}). Then, by  Theorem \ref{th5.ct.1}(ii), $u$ is a
duality solution to (\ref{eq1.1}). Consequently,  by Theorem
\ref{th5.ct.1}(i), $u$ is a strong duality solution to
\begin{equation}
\label{eq4.9.52}
-Au+u\cdot\nu=\mu_{|E_\nu}.
\end{equation}
Therefore $u$ is a strong duality solution to (\ref{eq1.1}), and
at the same time, a strong duality solution to  (\ref{eq4.9.52}).
By the definition of a strong duality solution and Remark
\ref{rem.5.321} we have $\langle\mu_{|E_\nu},R\beta\rangle=\langle
\mu,R\beta\rangle $ for every smooth measure $\beta$ such that
$R|\beta|$ is bounded. This implies that $\mu_{|E_\nu}=\mu$, so
$|\mu_c|(N_\nu)=0$.
\end{proof}

\section{Extension of the class $\MM_1$ and renormalized solutions}
\label{sec6}

In this section we show that the methods of proofs of existence results for (\ref{eq1.1}) achieved in the previous sections
in fact apply to the broader class of measures $\mu$ on the right-hand side of \eqref{eq1.1}. Denote by $\mathscr W$ the set
of  strictly positive, $m$-a.e. finite excessive functions on $E$. For $\rho\in\mathscr W$ we let  $\MM_\rho$ denote the set of Borel measures $\mu$ on $E$ such that 
\[
\int_E\rho(x)\,|\mu|(dx)<\infty,
\]
Taking $\rho\equiv 1$ we get $\MM_1$. 
Recall that if the Green function $G$ for $-A$ is strictly positive, then any function of the form $\rho=R\beta$
for non-trivial positive measure $\beta$ belongs to $\mathscr W$. In particular, if $\varphi$ is a strictly  positive principle eigenfunction
for $-A$,  then $\varphi=\lambda_1^{-1}R\varphi$ belongs to $\mathscr W$. 
Another  very important in applications class of functions included in $\mathscr W$ is the class of strictly positive harmonic functions.

\begin{lemma}
\label{lm.exospfmef13}
Let $\rho\in\mathscr W$. Then there exists a strictly positive function $g$
such that $Rg\le \rho$.
\end{lemma}
\begin{proof}
Clearly, without loss of generality we may assume that $\rho$ is bounded.
Let $\{U_n\}$ be an increasing sequence of relatively compact open subsets of $E$ such that $\bigcup_{n\ge 1} U_n=E$.
By \cite[Theorem 1.2.5, Lemma 2.3.5]{CF} for any $n\ge 1$ there exists a positive smooth measure $\nu_n$ such that $R\nu_n\le \rho$, and $R\nu_n=\rho$ q.e. on $U_n$.
Set $g_n:= R_1\nu_n$. Since $R\nu_n=\rho$ q.e. on $U_n$, we have that $R\nu_n>0$ q.e. on $U_n$, and hence $R_1\nu_n>0$ q.e. on $U_n$.
Thus, for any $n\ge 1$,  $g_n>0$ q.e. on $U_n$. Moreover,
\[
Rg_n=RR_1\nu_n=R_1R\nu_n\le R_1\rho\le \rho.
\]
Set $g=\sum_{n=1}^\infty 2^{-n}g_n$. Then $g$ is the desired function.
\end{proof}

\begin{remark}
\label{remark.com46}
For $\mu\in\MM_\rho$, we define duality solutions to (\ref{eq1.1}) as in Definition \ref{eq4.s51}
but with the class of test functions consisting of $\eta\in \BB(E)$ such that 
\begin{equation}
\label{eq.tfcbm4817}
R|\eta|\le c \rho,\quad \mbox{for some}\quad c>0.
\end{equation}
The only change in the definition of  strong duality solution (Definition \ref{def.sds3054}) to (\ref{eq1.1}) with  $\mu\in\MM_\rho$ is that
we require from $u$ to be in $L^1(E;\rho\cdot \nu)$.
The integrals in  (\ref{eq4.2s4}) are well defined and finite for q.e. $x\in E$.
Indeed, by Lemma \ref{lm.exospfmef13} there exists a strictly positive function $g$ such that $Rg\le\rho$.
Thus,
\[
\langle g,R|\mu|\rangle\le \langle \rho,|\mu|\rangle<\infty.
\]
Therefore, since $g$ is strictly positive, $R|\mu|<\infty,\,m$-a.e.
Since $R|\mu|$ is an excessive function, $R|\mu|<\infty$ q.e.
The same reasoning applies to $u\cdot\nu$.
Clearly, for $\mu\in\MM_\rho$, \eqref{eq5.3} holds for smooth $\beta$ such that 
\begin{equation}
\label{eq.tfcbm48173}
R|\beta|\le c \rho,\quad \mbox{for some}\quad c>0.
\end{equation}
Repeating step by step the proof of Theorem \ref{th4.s12}, Theorem \ref{th4.3.dsu}(i)-(ii), Theorem \ref{th5.ct.1}, and  Theorem \ref{th.main.ex.th23}
but using test functions/measures  satisfying (\ref{eq.tfcbm4817}), (\ref{eq.tfcbm48173}) we get assertions of these theorems for $\mu\in\MM_\rho$
(clearly, in Theorem \ref{th4.3.dsu}(ii) with $L^1(E;\nu)$ replaced by $L^1(E;\rho\cdot\nu)$).
\end{remark}

\begin{example}
To show an application of the Remark \ref{remark.com46} consider for a bounded open  $D\subset\BR^d$, $d\ge 2$, and $\alpha\in (0,2)$
\begin{equation}
\label{eq.tfcbm48174op}
A=(\Delta^{\alpha/2})_{|D},
\end{equation}
i.e. Dirichlet fractional Laplacian with zero exterior condition. Taking $\rho=\varphi_1^D$,
%Taking $\rho=R1$, we have (see e.g. \cite{Grzywny})
%that
%\begin{equation}
%\label{eq.tfcbm48174ad}
%c^{-1}\varphi^D_1\le \rho\le c \varphi_1^D
%\end{equation}
%for some $c>0$, 
where $\varphi_1^D$ is the principal eigenfunction of $(\Delta^{\alpha/2})_{|D}$,  we get the existence result for (\ref{eq1.1}) with $\mu$ in the class
\[
\MM_{\varphi_1^D}:=\{\mu\quad\mbox{is a Borel measure and}\quad \int_D\varphi_1^D\,d|\mu|<\infty\}.
\]
Furthermore, if $D$ is a $C^{1,1}$ domain, then it is well known (see e.g. \cite{Kulczycki}) that
\begin{equation}
\label{eq.tfcbm48174}
c^{-1}\delta_D^{\alpha/2}\le \varphi^D_1\le c \delta_D^{\alpha/2}
\end{equation}
for some $c>0$. Thus, we cover the class of measures considered in \cite{DGV}.  
It is worth mentioning  that $(\ref{eq.tfcbm48174})$ does not hold for arbitrary open domain $D$.
On the other hand, since $(P_t)_{t\ge 0}$ is {\em intrinsically ultracontractive} (see e.g. \cite{Grzywny}),
we have that
\begin{equation}
\label{eq.tfcbm48174gfi}
G(x,y)\ge c\varphi_1^D(x)\varphi_1^D(y),\quad x,y\in D.
\end{equation}
Therefore, we see that if $u$ is a solution to (\ref{eq1.1}) with $\nu\equiv 0$, and positive $\mu$, then
\[
u(x)=R\mu(x)\ge c\varphi_1^D(x)\int_D\varphi_1^D(y)\,\mu(dy),\quad\mbox{q.e.}
\]
So, the class $\MM_{\varphi_1^D}$ is optimal for problems of type (\ref{eq1.1})
with $A$ given by \eqref{eq.tfcbm48174op} and $\nu\equiv 0$ (cf. \cite[Proposition 3.10]{DGV}). 
In fact, by Theorem \ref{th4.s12} and \eqref{eq.tfcbm48174gfi}, the class $\MM_{\varphi_1^D}$ is optimal for \eqref{eq1.1}  with $A$ given by \eqref{eq.tfcbm48174op}, and $\nu$ satisfying
\[
c^{-1}G^\nu\le G\le c G^\nu,
\]
see \cite{Hansen} for the sufficient conditions on $\nu$ guaranteeing  the above comparability of $G$ and $G^\nu$.

\end{example}

We close this section with some comments on the notion of
renormalized solutions. For semilinear equations with
Dirichlet operators and general measure data this  notion was
introduced  in \cite{KR:MM}.   However, the concept of
renormalized solutions goes back to the paper by Dal Maso,  Murat,
Orsina and  Prignet \cite{DMOP}, where equations with local
nonlinear Leray-Lions type operators are considered.   In
\cite{KR:NoDEA} we observed that one of the equivalent formulation
of a renormalized solution to local equation with measure data
considered in  \cite{DMOP}  is also suitable for equations with
non-local operators and smooth measure data in the sense that it ensures uniqueness. In \cite{KR:MM} we generalized this result to non-local equations with general measure data.

The definition adopted in  \cite{KR:MM} reads as follows.
\begin{definition}
We say that a quasi-continuous function $u$  is a renormalized solution to (\ref{eq1.1}) if
\begin{enumerate}[(a)]
\item $u\in L^1(E;\nu)$, and $T_k(u)\in D_e(\EE)$, $k\ge 0$,
\item  There exists a sequence $\{\beta_k\}$ of bounded smooth measures such that for any bounded $\eta\in D_e(\EE)$ and any $k\ge 0$, 
\[
\EE(T_k(u),\eta)+\langle u\cdot\nu,\tilde \eta\rangle =\langle \mu_d,\tilde \eta\rangle+\langle\beta_k,\tilde\eta\rangle,
\]
\item $R\beta_k\rightarrow R\mu_c$ q.e. as $k\rightarrow \infty$.
\end{enumerate}
\end{definition}

By \cite{K:NoDEA}, under an additional assumption on the operator $A$, e.g. $R_1(\BB_b^+(E))\subset C_b(E)$,
the above definition is equivalent to the following one: $u$ is a quasi-continuous function satisfying (a), (b), and
\begin{enumerate}
\item[(c')] $\beta_k\rightarrow \mu_c$ narrowly as $k\rightarrow \infty$.
\end{enumerate}

\begin{proposition}
A quasi-continuous function  $u$ on $E$ is a strong duality solution to \mbox{\rm(\ref{eq1.1})} if and only if it is a renormalized solution to \mbox{\rm(\ref{eq1.1})}.
\end{proposition}
\begin{proof}
Follows from \cite[Theorem 4.4]{KR:MM} applied to (\ref{eq1.3}), i.e. to the linear equation with bounded measure $-u\cdot\nu+\mu$ on the right-hand side.
\end{proof}

\subsection*{Acknowledgements}
{\small This work was supported by Polish National Science Centre
(Grant No. 2017/25/B/ST1/00878).}

\end{document}